\documentclass[a4ppaper,11pt]{amsart}

\usepackage{enumerate}
\usepackage[latin1]{inputenc}
\usepackage{amsmath}
\usepackage{amsthm}
\usepackage{latexsym}
\usepackage{amssymb}
\usepackage{amsmath}
\usepackage{comment}
\usepackage{bm}
\usepackage{mathrsfs}
\usepackage{multicol}
\usepackage{abstract, lipsum}
\usepackage{etoolbox}
\usepackage{titlesec}
\usepackage{secdot}
\usepackage[all]{xy}
\usepackage{calc}
\usepackage{multicol}

\newtheorem{thm}{Theorem}[section]

\newtheorem{defi}{Definition}[section]
\newtheorem{lem}{Lemma}[section]
\newtheorem{cor}{Corollary}[section]

\newtheorem{rem}{Remark}[section]
\theoremstyle{notation}
\newtheorem*{notation}{Notation}

\newcommand{\R}{\mathbb{R}}
\numberwithin{equation}{section}

\newcommand{\eps}{\varepsilon}

\newcommand{\wto}{\rightharpoonup}

\makeatletter
\patchcmd{\@settitle}{\uppercasenonmath\@title}{}{}{}
\patchcmd{\@setauthors}{\MakeUppercase}{}{}{}
\patchcmd{\section}{\scshape}{}{}{}
\makeatother

\titleformat{\section}
 {\large\bfseries}{\thesection}{1em}{}

\usepackage[textwidth=138mm,
textheight=215mm,
left=30mm,
right=30mm,
top=25.4mm,
bottom=25.4mm,
headheight=2.17cm,
headsep=4mm,
footskip=12mm,
 heightrounded,]{geometry}

\usepackage[colorlinks=true,urlcolor=blue,
citecolor=black,linkcolor=black,linktocpage,pdfpagelabels,
bookmarksnumbered,bookmarksopen]{hyperref}
\usepackage[hyperpageref]{backref}

\makeatletter
\newcommand{\leqnomode}{\tagsleft@true}
\newcommand{\reqnomode}{\tagsleft@false}
\makeatother

\begin{document}

\title[]{Concentration Phenomenon of Semiclassical States to Reaction-Diffusion Systems}

\author{ {\normalsize{Tianxiang Gou}}${}^{a \ast}$, {\normalsize{Zhitao Zhang}}${}^{b,c,d}$\bigskip \\
${}^{a}${\it School of Mathematics and Statistics, Xi'an Jiaotong University, \\
Xi'an, Shaanxi 710049, China} \smallskip \\
${}^{b}${\it School of Mathematical Sciences, Jiangsu University, Zhenjiang, Jiangsu 212013, China} \smallskip \\
${}^{c}${\it HLM, Academy of Mathematics and Systems Science, Chinese Academy of Sciences, \\
Beijing 100190, China} \smallskip \\
${}^{d}${\it School of Mathematical Sciences, University of Chinese Academy of Sciences,\\
Beijing 100049, China} 
}

\thanks{ {\it E-mail addresses}: tianxiang.gou@xjtu.edu.cn (T. Gou), zzt@math.ac.cn (Z. Zhang)
\newline \indent Statements and Declarations. The author declares that there are no conflict of interests.
\newline \indent The work is supported by the National Key R \& D Program of China(2022YFA1005600), the National Science Foundation of China (Nos.12031015 and 12101483) and the Postdoctoral Science Foundation of China (No. 2021M702620). The authors would like to thank warmly the referees for the valuable suggestions and comments to improve the manuscript.}

\maketitle

\begin{abstract}

\vspace{0.2cm}

In this paper, we consider concentration phenomenon of semiclassical states to the following $2M$-component reaction-diffusion system in $\R \times \R^N$,
\begin{align*}
\left\{
\begin{aligned}
\partial_t u &=\eps^2 \Delta_x u-u-V(x)v + \partial_v H(u, v),\\
\partial_t v &=-\eps^2 \Delta_x v+v + V(x)u - \partial_u H(u, v),
\end{aligned}
\right.
\end{align*}
where $M \geq 1$, $N \geq 1$, $\eps>0$ is a small parameter, $V \in C^1(\R^N, \, \R)$, $H \in C^1(\R^M \times \R^M, \, \R)$ and $(u, v): \R \times \R^N \to \R^M \times \R^M$.
It is proved that there exist semiclassical states concentrating around the local minimum points of $V$ under mild assumptions. The approach is variational, which is mainly based upon a new linking-type argument, iterative techniques and interior estimates for nonlinear parabolic equations.
\medskip

 \medskip


\end{abstract}

\reqnomode

\section{Introduction}
In this paper, we are concerned with concentration phenomenon of semiclassical states to the following $2M$-component reaction-diffusion system in $\R \times \R^N$,
\begin{align} \label{system}
\left\{
\begin{aligned}
\partial_t u &=\eps^2 \Delta_x u-u-V(x)v + \partial_v H(u, v),\\
\partial_t v &=-\eps^2 \Delta_x v+v + V(x)u - \partial_u H(u, v),
\end{aligned}
\right.
\end{align}
where $M \geq 1$, $N \geq 1$, $\eps>0$ is a small parameter, $V \in C^1(\R^N, \, \R)$, $H \in C^1(\R^M \times \R^M, \, \R)$, and $(u, v): \R \times \R^N \to \R^M \times \R^M$. The system \eqref{system} arises in a wide variety of fields such as theoretical physics, chemistry and biology. It is generally applied to model the time variation of chemical concentrations due to reaction and diffusion. In such a system, $u$ and $v$ stand for chemical concentrations, the function $V$ describes a relative spatial distribution of chemical potential, and the nonlinear terms determined by the function $H$ represent external physico-chemical force, which govern dynamics of the system. The parameters $\eps^2$ and $-\eps^2$ are diffusion coefficients setting the pace of diffusion for chemicals $u$ and $v$, respectively. When diffusion coefficient is negative, which represents a phenomenon referred to as reverse diffusion. This often happens during phase separation, a situation where the transport of particles in a medium occurs towards regions of higher concentration. In addition, $\eps^2 \Delta_x u$ and $-\eps^2 \Delta_x v$ are called diffusion term and inverse diffusion term, respectively. The diffusion term specifies that $u$ increases in proportion to $\Delta_x u$, which indicates that when the quantity of $u$ is higher in neighboring areas,  $u$ will increase. Contrarily, the inverse diffusion term specifies that $v$ decreases in proportion to $\Delta_x v$, which indicates that when the quantity of $v$ is higher in neighboring areas,  $v$ will decrease. The nonlinearites $\partial_v H$ and $-\partial_u H$ are called reaction terms modeling chemicals reaction with a replenishment and diminishment, respectively. For more information regarding \eqref{system}, we refer the readers to \cite{Gr, Mu, Sm, Tur} and references therein.

We now recall some study in connection with \eqref{system}. To our knowledge, there are relatively few papers considering systems similar to \eqref{system}, most of which are indeed devoted to the study of the existence of solutions. In \cite{BrNi}, by using Schauder's fixed point theorem, the authors investigated the existence of positive solutions to the following $2$-component parabolic system in $(0, T) \times \Omega$,
\begin{align*}
\left\{
\begin{aligned}
\partial_t u &=\Delta_x u- v^5+f(x),\\
\partial_t v &=-\Delta_x v- u^3 + g(x),
\end{aligned}
\right.
\end{align*}
where $\Omega \subset \R^N$ is a bounded domain, $f, g \in L^{\infty}(\Omega)$, and $u(t, x)=v(t, x)=0$ for any $(t, x) \in (0, T) \times \partial \Omega$, $u(0, x)=v(T, x)=0$ for any $x \in \Omega$. Later, in \cite{ClFeMi}, via variational methods, the authors proved the existence of classical periodic and homoclinic solutions to the unbounded Hamiltonian system below set in $\R \times \Omega$,
\begin{align*}
\left\{
\begin{aligned}
\partial_t u &=\Delta_x u + |v|^{q-2}v,\\
\partial_t v &=-\Delta_x v - |u|^{p-2}u,
\end{aligned}
\right.
\end{align*}
where $\Omega \subset \R^N$ is a smooth bounded domain, $pq >1$, and $u(t, x)=v(t, x)=0$ for any $(t, x) \in \R \times \partial \Omega$.
Furthermore, in \cite{BaDi, Ding},  by establishing proper variational frameworks, the authors established the existence of homoclinic solutions to the following $2M$-component infinite dimensional Hamiltonian system in $\R \times \R^N$,
\begin{align*}
\left\{
\begin{aligned}
\partial_t u &=\Delta_x u-V(x)v + \partial_v H(t, x, u, v),\\
\partial_t v &=- \Delta_x v+ V(x)u - \partial_u H(t, x, u, v),
\end{aligned}
\right.
\end{align*}
where $V: \R^N \to \R$ is $1$-periodic in $x_j$ for any $j=1, \cdots, N$. We also refer the readers to \cite{DiLuWi} concerning the existence and multiplicity of homoclinic solutions to $2M$-component diffusion equations in $\R \times \Omega$, where $\Omega=\R^N$ or $\Omega \subset \R^N$ is a smooth bounded domain.

Regarding the further study related to \eqref{system}, it is worth mentioning \cite{DX, DX19}, where the authors discussed concentration of semiclassical states to \eqref{system} and proved there exist semiclassical states concentrating around the local minimum points of $V$. In fact, so far we are only aware of \cite{DX, DX19} considering this topic to \eqref{system}. 
The purpose of the present paper is to deeply explore the concentration of solutions to \eqref{system} under different context.

The research of the concentration of semiclassical states to nonlinear Schr\"odinger-type equations has attracted much attention in recent decades, there already exists a great deal of literature. However, much less is known to \eqref{system}. By the well-known Lyapunov-Schmidt reduction technique, the authors in \cite{FlWe} first proved that there exists a single spike semiclassical state to the following equation with $N=1$ and $f(w)=|w|^2 w$,
\begin{align} \label{NLS}
-\eps^2 \Delta w + V(x) w = f(w) \quad \mbox{in} \,\, \R^N.
\end{align}
It also turns out that the solution concentrates around any given non-degenerate critical point of the potential $V$. The result was extended by the author in \cite{Oh1, Oh2} to the case $N \geq 2$ and $f(w)=|w|^{p-2}w$ for $2 < p < 2^*$.
Afterwards, utilizing minimax arguments, the author in \cite{Ra} considered the existence of semiclassical states to \eqref{NLS} under the assumption
\begin{align} \label{glovp}
\inf_{x \in \R^N} V(x)<\lim_{|x| \to \infty} V(x).
\end{align}
In \cite{Wang}, the author further addressed that there exist semiclassical states to \eqref{NLS} concentrating around the global minimum points of the potential $V$. Subsequently, in \cite{ByJe, ByJeTa, ByWa1, ByWa2, deFe, deFe1, MoVa1}, the concentration of semiclassical states to \eqref{NLS} around the local minimum points of the potential $V$ was discussed under the assumption
\begin{align} \label{locvp}
\inf_{x \in \Lambda } V(x) < \inf_{x \in \partial \Lambda} V(x),
\end{align}
where $\Lambda \subset \R^N$ is a bounded domain. We also refer the readers to \cite{AmBaCi, ByTa1, ByTa2, CadoMo, ChLiWa, ChWa, DadeWe, FiIkJo, MoVa2, RuVa} and references therein for the relevant survey.

\subsection*{Statement of main result.} In order to state our main result, we now show assumptions imposed on $V$ and $H$. For the potential $V$, we make the following assumptions,
\begin{enumerate}
\item [($V_1$)] $V \in C^1(\R^N, \R)$ and $\|V\|_{\infty}:=\sup_{x \in \R^N} |V(x)| <1$;
\item [($V_2$)] there exists a bounded domain $\Lambda \subset \R^N$ with smooth boundary $\partial \Lambda$ such that
$$
\nabla V(x) \cdot {\bf{n}}(x) >0 \quad \mbox{for any} \ \, x \in \partial \Lambda,
$$
where ${\bf{n}}(x)$ denotes the unit outward normal vector at $x \in \partial \Lambda$.
\end{enumerate}

\begin{rem}
Note that $(V_2)$ is satisfied if $V$ has an isolated local minimizers set, i.e., $V$ has a locally trapping potential well. Such an assumption on $V$ is more general than the usual ones \eqref{glovp} and \eqref{locvp}, which indeed makes our study different from the one conducted in \cite{DX}.
\end{rem}


For the  nonlinear function $H$, we assume that
$$
H(z)=G(|z|):=\int_{0}^{|z|}g(s)s \, ds \quad \mbox{for any} \,\, z \in \R^M \times \R^M,
$$
where $g$ fulfills the following assumptions,
\begin{enumerate}
\item [$(H_1)$] $g \in C(\R^+, \, \R^+) \cap C^1((0, \infty), \, \R^+)$ and $g(0)=0$, where $\R^+:=[0, \infty)$;
\item [$(H_2)$] there exist $c>0$ and $2 <p < 2(N+2)/N$ such that $g(s) \leq c(1 + s^{p-2})$ for any $ s \geq 0$;
\item [$(H_3)$] $\lim_{s \to \infty}\frac{G(s)}{s^2}=\infty$;
\item [$(H_4)$] $g$ is nondecreasing on $[0, \infty)$.
\end{enumerate}

\begin{rem}
Note that, in our case,  there holds that
$$
\frac 12 g(s)s^2 - G(s) \geq 0 \quad \mbox{for any} \,\, s \geq 0.
$$
The classical Ambrosetti-Rabinowitz condition is not required.  The assumptions on the nonlinear function $H$ are rather weak to guarantee the existence of ground states to \eqref{system}.
\end{rem}

It is simple to see that the assumptions $(H_1)$-$(H_4)$ are satisfied by a large class of functions. Two typical examples are $g(s)=\ln(1 +s)$ and $g(s)=s^{p-2}$ for any $2 <p < 2(N+2)/N$ and $s \geq 0$.

Let us next fix some notations. Under the assumption $(V_2)$, the set of critical points of $V$ is defined by
\begin{align} \label{defv}
\mathcal{V}:=\{x \in \Lambda: \nabla V(x)=0\}.
\end{align}
Clearly, $\mathcal{V}$ is a nonempty compact subset of $\Lambda$. Without loss of generality, we shall assume that $0 \in \mathcal{V}$. For any set $\Omega \subset \R^N$, $\eps>0$ and $\delta >0$, we define that
$$
\Omega_{\eps}:=\left\{x \in \R^N: \eps x \in \Omega\right\},
$$
and
$$
\Omega^{\delta}:=\left\{x \in \R^N: \mbox{dist}(x, \, \Omega):=\inf_{y \in \Omega} |x-y| < \delta\right\}.
$$

The main result of this paper reads as follows.

\begin{thm} \label{theorem}
Suppose that $(V_1)$-$(V_2)$ and $(H_1)$-$(H_4)$ hold, then there exists a constant $\eps_0>0$ such that, for any $0<\eps <\eps_0$, \eqref{system} admits a ground state $z_{\eps}:=(u_{\eps}, v_{\eps})$ satisfying that, for any $\delta>0$, there exist $c=c(\delta)>0$ and $C=C(\delta)>0$ such that
$$
\left|z_{\eps}(t, x) \right| \leq C \, \textnormal{exp}\left(-\frac{c\, \textnormal{dist}\left(x, \mathcal{V}^{\delta}\right)}{\eps}\right).
$$
\end{thm}

The result provides a characterization of concentration phenomenon of chemicals. It reveals that chemicals concentrate around the local minimum points of the spatial distribution of chemical potential for small diffusion coefficients.

In \cite{DX}, the authors investigated the concentration of semiclassical states to \eqref{system} under the assumption \eqref{locvp} instead of $(V_2)$, under which the associated autonomous systems do exist and they play an essential role in the discussion.
However, in our situation, under the assumption $(V_2)$, there do not exist the associated autonomous systems to use, which makes the abstract critical point theorem obtained in \cite{DX} cannot be adapted to our problem. For this reason, we establish a new linking-type argument to derive the desired existence result. Moreover, our assumptions on the nonlinearity are weaker than the ones in \cite{DX}.
And we derive the exponential decay of semiclassical states to \eqref{system}, which was not given in \cite{DX}.

We now sketch the outline of the proof of Theorem \ref{theorem}. To begin with, by making a change of variable $ x \to \eps x$, we introduce an equivalent system \eqref{system1}.
It is standard that any solution to \eqref{system1} corresponds to a critical point of the underlying energy functional $J_{\eps}$ defined by \eqref{defof}. Note that the functional $J_{\eps}$ does not satisfy the desired compactness condition. Thus we are unable to directly rely on the  functional $J_{\eps}$ to seek for ground states to \eqref{system1}. Thereby, a modified energy functional $\Phi_{\eps}$ defined by \eqref{defmf} is introduced. At this point, in order to complete the proof of this theorem, we take the following two crucial steps. 

{Step 1} : Prove that, for any $\eps>0$ small, the functional $\Phi_{\eps}$ possesses nontrivial critical points minimizing the functional $\Phi_{\eps}$ among all its critical points, which are indeed ground states to \eqref{equation}. To achieve this, we shall bring in the generalized Nehari manifold corresponding to the functional $\Phi_{\eps}$ and demonstrate the existence of minimizers to the functional $\Phi_{\eps}$ subject to the manifold. Notice that, in our scenario, for any $x \in \R^N$, the modified nonlinear function $f_{\eps}(x, \cdot)$ defined by \eqref{deff} is only nondecreasing but not strictly increasing on $[0, \infty)$. In addition, the functional $\Phi_{\eps}$ is not $\mathcal{T}$-upper semicontinuous, where the topology $\mathcal{T}$ is induced by the norm given by \eqref{deftopy}. This enables that the approaches developed in \cite{KrSz, Pan, SzWe} used to investigate the existence of solutions to strongly indefinite problems are not directly applicable to our problem. To overcome this difficulty, we need to borrow ideas from \cite{Med}, where the author succeeded in attaining the existence of ground states to strongly indefinite problem without imposing the strict monotonicity condition on the nonlinearity. However, $\mathcal{T}$-upper semicontinuity of the energy functional is required there.
Therefore, we also need to employ elements from \cite{ChXi, ChWang, GuZh, KrSz}, where the existence of solutions to strongly indefinite problems was considered without imposing $\mathcal{T}$-upper semicontinuity assumption. Let us remark that the adaption of ingredients from the existing literature to our problem is highly nontrivial, because we work in distinctive setting. At this moment, we are able to establish a new linking-type argument to our problem, see Lemma \ref{deformation}, by which the desired existence result follows necessarily.
The argument we establish is new and extends the previous ones, which holds under more general conditions and may be applicable to other problems.

{Step 2} : Prove that, for any $\eps>0$ small, ground states to \eqref{equation} decay exponentially, from which ground states to \eqref{equation} are indeed ones to \eqref{system1} with the desired decay. To attain this, we make use of the well-known Lions concentration compactness lemma, see Lemma \ref{concentration}, and the iteration techniques developed in \cite{ChWa} along with interior estimates for nonlinear parabolic equations. Let us point out that, under our circumstance, the proof of the exponential decay requires more delicate analysis, because we are concerned with $2M$-component reaction-diffusion systems instead of nonlinear elliptic equations, which are parabolic systems set on $t$-Anisotropic Sobolev spaces.

\subsection*{Structure of the paper.} The remainder of the paper is laid out as follows. In Section \ref{pre}, we shall establish the associated variational frameworks for our problem and present some crucial lemmas used frequently in our proofs. Section \ref{pf} is devoted to the proof of Theorem \ref{theorem}, which is divided into two parts. In the first part, we shall prove the existence of ground states. In the second part, we shall deduce exponential decay of ground states.

\begin{notation}
Throughout the paper, for any $1  \leq q \leq \infty$ and $n \in \mathbb{N}^+$ with $n \geq 1$, we denote by $L^q(\R^n)$ the usual Lebesgue space and denote by $W^{1, q}(\R^n)$ and $W^{2, q}(\R^n)$ the usual Sobolev spaces. We use the notations $o_n(1)$ and $o_{\eps}(1)$ for quantities which tend to zero as $n \to \infty$ and $\eps \to 0^+$, respectively. For any $T, R>0$, $B(\tau, T)$ denotes the open ball in $\R$ with center at $\tau \in \R$ and radius $T$, and $B(y, R)$ denotes the open ball in $\R^N$ with center at $y \in \R^N$ and radius $R$. Furthermore, $\partial B(y, R)$ denotes the sphere of $B(y, R)$. We write $\overline{Q}$ for the closure of a set $Q \in \R^n$. We use letters $c$ and $C$ for generic positive constants, whose values may change from line to line.
\end{notation}

\section{Preliminary results} \label{pre}

In this section, we shall present some preliminary results used to establish our main result. To begin with, by making a change of variable $ x \to \eps x$, we see that \eqref{system}  becomes
\begin{align} \label{system1}
\left\{
\begin{aligned}
\partial_t u &=\Delta_x u-u-V_{\eps}(x)v + \partial_v H(u, v),\\
-\partial_t v &=\Delta_x v-v-V_{\eps}(x)u + \partial_u H(u, v),
\end{aligned}
\right.
\end{align}
where $V_{\eps}(x):=V(\eps x)$. Set
\begin{align*}
\mathcal{J}:= \left(
\begin{array}{cc}
  0 & -I \\
  I & 0
\end{array}
\right),  \quad
\mathcal{J}_0:= \left(
\begin{array}{cc}
  0 & I \\
  I & 0
\end{array}
\right),
\quad A:=\mathcal{J}_0 \left(-\Delta_x + 1 \right),
\end{align*}
and
\begin{align}\label{defl}
L:=\mathcal{J} \partial_t + A,
\end{align}
then \eqref{system1} may be written as
$$
L z + V_{\eps}(x) z = g(|z|) z \quad \mbox{for} \,\, z:=(u, v).
$$

\subsection{Functional settings}

For any $1 \leq q  \leq \infty$, we denote by $L^q:=L^q(\R \times \R^N, \, \R^{2M})$ the usual Lebesgue space equipped with the norm $\|\cdot\|_q$. Notice that $L$ acting on $L^2$ is a self-adjoint operator with domain
$$
D(L):=W^{1, 2}\left(\R, L^2(\R^N, \R^{2M})\right) \cap L^2\left(\R, W^{2,2}\left(\R^N, \R^{2M}\right) \right).
$$

\begin{lem} \label{spectrum} \cite[Lemma 8.7]{Ding}
Assume that $(V_1)$ holds, then $\sigma(L)=\sigma_e(L) \subset \R \backslash (-1, 1)$, where $\sigma(L)$ and $\sigma_e(L)$ denote the spectrum and essential spectrum of $L$, respectively.
\end{lem}

Let $\{E_{\lambda}\}_{\lambda \in \R}$ be the spectrum family of $L$. According to \cite[Chapter IV, Theorem 3.3]{EE}, $L$ admits the polar decomposition
\begin{align} \label{polar}
L=U|L|=|L|U,
\end{align}
where $U$ is a unitary isomorphism of $L^2$ such that $U=I-2E_0$, and $|L|$ denotes the absolute value of $L$. This, along with Lemma \ref{spectrum}, suggests that $L^2$ possesses an orthogonal decomposition
\begin{align*} 
L^2=L^+ \oplus L^{-}
\end{align*}
such that $L$ is positive definite on $L^+$ and negative definite on $L^{-}$, where
\begin{align} \label{decom}
L^\pm:=\left\{z \in L^2: Uz=\pm z \right\}.
\end{align}
In order to seek for solutions to \eqref{system1}, let us introduce $E:=D(|L|^{\frac 12})$ with the inner product
$$
\langle z_1, \, z_2\rangle:=(|L|^{\frac 12} z_1, \,|L|^{\frac 12} z_2)_2 \quad \mbox{for any} \, \, z_1, z_2 \in E,
$$
where $(\cdot, \cdot)_2$ stands for the usual inner product in $L^2$, and $|L|^{\frac 12}$ denotes the square root of $L$.  For any $z \in E$, the induced norm $\|z\|:=\langle z, \, z \rangle^{\frac 12}.$ Clearly, $E$ is a Hilbert space. By the interpolation theory in \cite{Tr}, one can see that $E=[D(L), \,L^2]_{1/2}$. 


\begin{lem} \label{embedding} \cite[Lemma 4.6]{BaDi}
Assume that $(V_1)$ holds, then $E$ is continuously embedded in $L^r$ for any $r \in [2, \infty)$ if $N \geq 1$, and for any $r \in [2, 2(N+2)/N]$ if $N \geq 2$. $E$ is compactly embedded in $L_{loc}^r$ for any $r \in [1, \infty)$ if $N \geq 1$, and for any $r \in [1, 2(N+2)/N)$ if $N \geq 2$.
\end{lem}

From the orthogonal decomposition to $L^2$, the space $E$ admits the following associated decomposition
$$
E=E^+ \oplus E^{-},
$$
where $E^{\pm}:=E \cap L^{\pm}$. The decomposition is orthogonal with respect to $(\cdot, \cdot)_2$ and $\langle \cdot, \cdot \rangle.$ In fact, for any $z^+ \in E^+$ and $z^- \in E^-$, we know that $z^+ \in L^+$ and $z^- \in L^-$, then $(z^+, \, z^-)_2=0$. Note that
\begin{align} \label{orthogonal}
\begin{split}
\langle z^+, \,z^-\rangle &= (|L|^{\frac 12} z^+, \,|L|^{\frac 12} z^-)_2=(|L|z^+, \,z^-)_2=(|L|U z^+, \,z^-)_2 \\
&=(L z^+, \,z^-)_2=(z^+, \,Lz^-)_2=(z^+, \,|L|Uz^-)_2  \\
&=-(z^+, \,|L|z^-)_2=-(|L|^{\frac 12} z^+, \,|L|^{\frac 12} z^-)_2 \\
&=-\langle z^+, z^- \rangle,
\end{split}
\end{align}
where we used the polar decomposition and self-adjointness of $L$.  Accordingly, \eqref{orthogonal}  readily infers that $\langle z^+, z^- \rangle =0$. As a result, for any $z \in E$,
\begin{align*}
(Lz,\, z)_2&=(Lz^+ + Lz^-, \, z^+ + z^-)_2=(Lz^+, \, z^+ + z^-)_2 +(Lz^-, \, z^+ + z^-)_2\\
&=(|L|U z^+, \, z^+ + z^-)_2 +(|L|U z^-, \, z^+ + z^- )_2\\
&=(|L| z^+, \, z^+ + z^-)_2 - (|L| z^-, \, z^+ + z^- )_2 \\
&=(|L|^{\frac 12} z^+, \,|L|^{\frac 12} z^+ + |L|^{\frac 12} z^-)_2-(|L|^{\frac 12} z^-, \,|L|^{\frac 12} z^+ + |L|^{\frac 12} z^-)_2\\
&=\langle z^+, z^+ \rangle -\langle z^-, z^- \rangle \\
&=\|z^+\|^2-\|z^-\|^2,
\end{align*}
from which the energy functional associated to \eqref{system1} is given by
\begin{align}\label{defof}
J_{\eps}(z):= \frac 12 \left(\|z^+\|^2-\|z^-\|^2 \right) + \frac 12 \int_{\R} \int_{\R^N} V_{\eps}(x)|z|^2 \, dt dx - \int_{\R} \int_{\R^N} G(|z|)\, dt dx.
\end{align}
It follows from $(H_1)$ and $(H_2)$ that there exist $c_1, c_2 >0$ such that
$$
G(s)\leq c_1 s^2 + c_2 s^{p} \quad \mbox{for any} \ \, s \geq 0.
$$
Then, in view of Lemma \ref{embedding}, the functional $J_{\eps}$ is well-defined on $E$. Moreover, it is of class $C^1$, and for any $w \in E$,
$$
J_{\eps}'(z)w = \int_{\R}\int_{\R^N} Lz \cdot w \, dtdx +  \int_{\R} \int_{\R^N} V_{\eps}(x)z \cdot w \, dt dx - \int_{\R} \int_{\R^N} g(|z|) z \cdot w\, dt dx,
$$
which reveals that critical points of $J_{\eps}$ are solutions to \eqref{system1}.

In order to discuss the concentration of semiclassical states to \eqref{system}, we need to introduce a modified functional on $E$. To do this, let us first show some notations. According to $(V_2)$, we know that there is $\delta_0 >0$ such that, for any $y \in \Lambda^{\delta_0}$,  if $B(y, \, \delta_0) \backslash \Lambda \neq \emptyset$, there holds that
\begin{align} \label{vloc}
\inf_{x \in B(y, \, \delta_0) \backslash \Lambda} \nabla V(x)\cdot \nabla \mbox{dist}(x, \Lambda)>0.
\end{align}
Let $\zeta \in C^{\infty}(\R, [0, 1])$ be a cut-off function with $\zeta(t)=0$ if $t \leq 0$, $\zeta(t) >0$ if $t >0$ and $\zeta(t) =1$ if $t \geq \delta_0$, and $\zeta'(t) \geq 0$ for any $t \geq 0$. Set $\chi(x):=\zeta(\mbox{dist}(x, \, \Lambda))$ and
\begin{align}\label{tildeg}
\tilde{g}(s):=\min\left\{g(s), \, \mu\right\}, \quad \tilde{G}(s):=\int_{0}^{s} \tilde{g}(\tau) \tau \, d\tau \quad \mbox{for any} \,\, s \geq 0,
\end{align}
where $\mu:= \frac{1-\|V\|_{\infty}}{2}.$ For any $x \in \R^N$ and $s \geq 0$, we now define that
\begin{align} \label{deff}
f(x, s):=\left(1- \chi(x)\right) g(s)+\chi(x)\tilde{g}(s), \quad  F(x,s):=\int_{0}^s f(x, \tau) \tau \, d\tau,
\end{align}
then the modified functional on $E$ is introduced as
\begin{align} \label{defmf}
\hspace{-1.0cm} \Phi_{\eps}(z):= \frac 12 \left(\|z^+\|^2-\|z^-\|^2\right)+ \frac 12 \int_{\R} \int_{\R^N} V_{\eps}(x)|z|^2 \, dtdx - \int_{\R}\int_{\R^N} F_{\eps}(x, |z|) \, dtdx,
\end{align}
where $F_{\eps}(x, |z|):=F(\eps x, |z|)$. As a consequence of $(H_1)$ and $(H_2)$, we know that, for any $\gamma>0$, there exists $c_{\gamma}>0$ such that
\begin{align} \label{h1h21}
f_{\eps}(x, s) \leq \gamma  + c_{\gamma}s^{p-2} \quad \mbox{for any} \,\, x \in \R^N, s \geq 0.
\end{align}
This then indicates that
\begin{align} \label{h1h2}
F_{\eps}(x, s) \leq \frac{\gamma}{2} s^2 + \frac {c_{\gamma}}{p} s^p \quad \mbox{for any} \,\, x \in \R^N, s \geq 0.
\end{align}
Plainly, by Lemma \ref{embedding}$, \Phi_{\eps}$ is well-defined on $E$, and it is of class $C^1$. Furthermore, for any $w \in E$, we have that
$$
\Phi_{\eps}'(z)w = \int_{\R}\int_{\R^N} Lz \cdot w \, dtdx +  \int_{\R} \int_{\R^N} V_{\eps}(x)z \cdot w \, dt dx - \int_{\R} \int_{\R^N} f_{\eps}(x, |z|) z \cdot w\, dt dx,
$$
where $f_{\eps}(x, |z|):=f(\eps x, |z|)$. Thus critical points of $\Phi_{\eps}$ are solutions to the system
\begin{align} \label{equation}
L z +V_{\eps}z=f_{\eps}(x, |z|)z.
\end{align}
Additionally, from \eqref{deff} and $(H_4)$, it is easy to see that
\begin{align} \label{notar}
\frac 12 f(x,  s)s^2 - F(x, s) \geq 0 \quad \mbox{for any} \,\, x \in \R^N, s \geq 0.
\end{align}

\subsection{Some key lemmas}

In what follows, we shall present some lemmas to be used frequently throughout the paper, which play an important role in our proofs.

\begin{lem} \label{l2}
Assume that $(V_1)$ holds, then $\|z\|_2 \leq \|z\|$.
\end{lem}
\begin{proof}
Since $\sigma (L) \subset \R \backslash (-1, 1)$, see Lemma \ref{spectrum}, it then follows from the operator spectrum theory that
\begin{align*}
\|z\|^2=\langle |L|^\frac 12  z, |L|^{\frac 12} z\rangle &= \int_{1}^{\infty} |\lambda|^{\frac 12} \, d(E_{\lambda} z, z)_2 + \int_{-\infty}^{-1} |\lambda|^{\frac 12} \, d(E_{\lambda}z, z)_2 \\
& \geq \int_{1}^{\infty} d(E_{\lambda} z, z)_2 + \int_{-\infty}^{-1} d(E_{\lambda}z, z)_2 \\
&=\|z\|_2^2,
\end{align*}
and the conclusion follows.
\end{proof}

The orthogonal decomposition of $E$ induces a natural decomposition of $L^q$, and we have the following result, see also \cite{DX19}.

\begin{lem} \label{lp}
Assume that $(V_1)$ holds, then $\|z^{\pm}\|_q \leq c_q \|z\|_{q}$ for any $ 2 \leq q  \leq {2(N+2)}/N $.
\end{lem}

In order to prove this lemma, let us introduce the definition of multiplier.

\begin{defi}
Let m be a bounded measurable function on $\R^n$, and define a linear operator $T_m$ on $L^q(\R^n) \cap L^2(\R^n)$ by
$$
{\widehat{T_m u}}(\xi):=m(\xi)\hat{u}(\xi),
$$
where $\hat{u}$ denotes the Fourier transform of $u$, and $1 \leq q \leq \infty$. We say that $m$ is a multiplier for $L^q(\R^n)$, if $T_m u \in L^q(\R^n)$ for any $u \in L^q(\R^n) \cap L^2(\R^n)$, and $T_m$ is bounded, i.e.
\begin{align} \label{samebdd}
\|T_m u\|_{L^q} \leq C \|u\|_{L^q} \quad \mbox{for any} \,\, u \in L^q(\R^n) \cap L^2(\R^n),
\end{align}
where $C>0$ is independent of $u$, and $\|\cdot\|$ denotes the norm in $L^q(\R^n)$.
\end{defi}

\begin{rem}
For any $ 1 \leq q < \infty$, by the denseness of $L^q(\R^n) \cap L^2(\R^n)$ in $L^q(\R^n)$, we know that $T_m$ has a unique bounded extension to $L^q(\R^n)$ satisfying the same inequality \eqref{samebdd} for any $ u \in L^q(\R^n)$.
\end{rem}

With this definition in hand, we are now ready to present the proof of Lemma \ref{lp} for convenience of readers.

\begin{proof} [Proof of Lemma \ref{lp}]
The proof of this lemma is inspired by the one of \cite[Proposition 2.1]{DX1}. By the definition of $L$,  in the Fourier domain $\xi:=(\xi_0, \xi_1, \cdots, \xi_N) \in \R \times \R^N$, $L$ becomes the operator of multiplication by the matrix
$$
\hat{L}(\xi):= \left(
\begin{array}{cc}
  0 & \left(-2\pi i \xi_0 + 4 \pi^2 \sum_{k=1}^N |\xi_k|^2 + 1\right) I \\
 \left (2\pi i \xi_0 + 4 \pi^2 \sum_{k=1}^N |\xi_k|^2 + 1\right) I & 0
\end{array}
\right),
$$
where $I$ is the $M \times M$ identity matrix. It is straightforward to compute that $\hat{L}(\xi)$ has two eigenvalues $\lambda_1, \lambda_2 \in \R$ with
$$
\lambda_1=\sqrt{4 \pi^2 |\xi_0|^2 + \left(1 + 4 \pi^2 \sum_{k=1}^N |\xi_k|^2\right)^2}, \quad
\lambda_2=-\sqrt{4 \pi^2 |\xi_0|^2 + \left(1 + 4 \pi^2 \sum_{k=1}^N |\xi_k|^2\right)^2}.
$$
We now denote by $P$ the projection operator on $E^+$ such that $P u=u^+$ for any $u \in E$. Note that $P$ admits the following representation,
$$
P=\frac{|L|^{-1}}{2}\left (|L| + L\right),
$$
which is a straightforward consequence of \eqref{polar} and \eqref{decom}. Consequently, in the Fourier domain, $P$ is a multiplication operator by a bounded smooth matrix-valued function $m(\xi)$, i.e.
$$
{\widehat{P u}}(\xi)=m(\xi) \hat{u}(\xi),
$$
where
$$
m(\xi):= \left(
\begin{array}{cc}
 \frac 12 I & \frac{1}{\lambda_1}\left(-\pi i \xi_0 + 2\pi^2 \sum_{k=1}^N |\xi_k|^2 + \frac 12 \right) I \\
 \frac{1}{\lambda_1} \left (\pi i \xi_0 + 2 \pi^2 \sum_{k=1}^N |\xi_k|^2 + \frac 12\right) I &\frac 12 I
\end{array}
\right).
$$
At this point, we are able to apply Marcinkiewicz multiplier theorem,  see \cite[Chapter 4, Theorem 6]{Stein}, to conclude that $P$ is a multiplier for $L^q$, which then implies that $\|u^+\|_q \leq c_q \|u\|_q$. Analogously, we can prove that $\|u^-\|_q \leq c_q \|u\|_q$. Hence the proof is completed.
\end{proof}

\begin{rem} \label{remlp}
If $q=2$,  then $\|z^{\pm}\|_2 \leq \|z\|_2$. Indeed, for any $z \in L^2$,  using the orthogonality of the decomposition in $L^2$, we obtain that
$$
\|z\|^2_2=(z, \, z)_2=(z^+ + z^-, \, z^+ + z^-)_2=(z^+, \, z^+)_2 + (z^-, \, z^-)_2=\|z^+\|_2^2 + \|z^-\|^2_2,
$$
where $z^{\pm} \in L^2$, the conclusion then follows.
\end{rem}

We next give so-called Lions' concentration compactness lemma in $E$.

\begin{lem} \label{concentration}
Let $T, R>0$. If $\{z_n\} \subset E$ is bounded, and
$$
\sup_{(\tau, \, y) \in \R \times \R^N}\int_{B(\tau, \,T)}\int_{B(y, \, R)} |z_n|^2 \, dtdx =o_n(1),
$$
then $z_n \to 0$ in $L^q$ for any $2 <q<2(N+2)/2$.
\end{lem}
\begin{proof}
The proof of this lemma is almost identical to the one of the classical Lions' concentration compactness lemma \cite[Lemma I.1]{Lions}, hence we omit it.
\end{proof}

In the following, we show two crucial lemmas from \cite{DX}.

\begin{lem} \label{regularity} \cite[Lemma A.5]{DX}
Let $V \in L^{\infty}(\R \times \R^N, \, \mathcal{M}_{2M \times 2M})$ and $H: \R \times \R^N \times \R^{2M} \to \R$ satisfy
$$
\left|\nabla_{z}H(t,\, x,\, z)\right| \leq |z| + c|z|^{p-1}
$$
for some $c>0$ and $2 < p <2(N+2)/N$. If $z \in E$ is a weak solution to the system
$$
L z + V(t, x)z=\nabla_z H(t, x, z),
$$
then $z \in B^q$ for any $q \geq 2$, and
$$
\|z\|_{B^q} \leq C(\|M\|_{\infty}, \, \|z\|, \,c, \, p,\, q),
$$
where $\mathcal{M}_{2M \times 2M}$ denotes the space of $2M \times 2M$ real matrixes equipped with the usual vector norm, $L$ is defined by \eqref{defl}, and the Banach space
\begin{align} \label{defbq}
B^{q}:=W^{1, q}(\R, \,L^q(\R^N,\, \R^{2M})) \cap L^{q}(\R, \, W^{2, q}(\R^N, \, R^{2M}))
\end{align}
with the usual norm
\begin{align*}
\|z\|_{B^q}:=\left(\int_{\R}\int_{\R^N} |z|^q+|\partial_t z|^q+|\nabla z|^q + \sum_{1 \leq i,\,j \leq N} |\partial_{i, j}z|^q  \, dtdx \right)^{1/q}.
\end{align*}
\end{lem}

\begin{lem} \cite[Corollary A.4]{DX} \label{est2.7}
Let $\frac{N+2}{2}<q< \infty$, $r>0$, and set $Q_{r}(t, x):=(-r^2, t]\times B(x, r)$. If $ w \in L^q(Q_r(t, x))$ is a weak solution to
\begin{align} \label{equu}
\partial_t w-\Delta w+w= h
\end{align}
with $h \in L^q(Q_{r}(t, x))$, then, for any $0<\sigma<r$,
$$
\|w\|_{C^{\alpha, \, \alpha/2}(\overline{Q_{r-\sigma}(t, \, x)})} \leq C(N, q, r, \sigma) \left(\|h\|_{L^q(Q_r(t, \, x))}+ \|w\|_{L^q(Q_r(t,\,x))}\right),
$$
where $0<\alpha \leq 2-\frac{N+2}{q}$, and
$$
\|w\|_{C^{\alpha/2, \, \alpha}(\overline{\mathcal{Q}})}:=\sup_{(t, \,x) \in \overline{\mathcal{Q}}}|w(t, \,x)|+\sup_{\tiny{\begin{array}{c}
(t_1, \,x_1),(t_2, x_2) \in \overline{\mathcal{Q}} \\
(t_1,\, x_1) \neq (t_2, \, x_2)\end{array}}} \frac{|w(t_1,\, x_1)-u(t_2,\, x_2)|}{\mbox{d}^{\, \alpha}\left((t_1, \,x_1), \,(t_2,\, x_2)\right)},
$$
for $\mathcal{Q}:=(a, b) \times \Omega$ with $a, b \in \R$, $a<b$, $\Omega \subset \R^N$, and
$$
\mbox{d}\left((t_1, \,x_1), \,(t_2,\, x_2)\right)=\max \{|t_1-t_2|^{1/2}, \,|x_1-x_2|\}.
$$
\end{lem}

For our purpose, we require the following interior estimate.

\begin{cor} \label{estimate}
Under the assumptions of Lemma \ref{estimate}, we have that
$$
\|w\|_{C (\overline{Q_{r-\sigma}(t,\, x)})}\leq C(N, q, r, \sigma) \left(\|h\|_{L^q(Q_r(t, \, x))}+ \|w\|_{L^q(Q_r(t,\,x))}\right),
$$
where
$$
\|w\|_{C(\overline{\mathcal{Q}})}:=\sup_{(t, \,x) \in \overline{\mathcal{Q}}}|w(t, \,x)|.
$$
\end{cor}

\section{Proof of main result} \label{pf}

In this section, our aim is to prove Theorem \ref{theorem}. From now on, we always assume that $(V_1)$-$(V_2)$ and $(H_1)$-$(H_4)$ hold.

\subsection{Existence of ground states} We first consider the existence of ground states to \eqref{equation}. To do this, let us introduce the following generalized Nehari manifold associated to \eqref{equation},
$$
\mathcal{N}:=\left\{z \in E \setminus E^-: \Phi'_{\eps}(z)z=0 \,\, \mbox{and} \,\, \Phi'_{\eps}(z) w=0 \,\, \mbox{for any} \,\, w\in E^-\right\}.
$$
This type of manifold was initially proposed in \cite{Pan} and deeply studied in \cite{SzWe}. For any $z \in E \setminus E^-$, let us define that
$$
\hat{E}(z):= \R^+ z^+ + E^-.
$$
Moreover, for any $z \in E \setminus E^-$, we define a functional $\gamma_{\eps, z}: \R^+ \times E^- \to \R$ by
$$
\gamma_{\eps, z}(\tau, \, w):=\Phi_{\eps}(\tau z^+ +w).
$$
Obviously, $\gamma_{\eps, z}$ is of class $C^1$. \medskip

We next show some basic properties related to the manifold $\mathcal{N}$, which lay a foundation to establish the existence of ground states to \eqref{equation}.

\begin{lem} \label{critical}
For any $z \in E \setminus E^-$, $(\tau, \, w)$ is a critical point of $\gamma_{\eps, z}$ if and only if $\tau z^+ + w \in \mathcal{N}$, where $\tau>0$ and $w \in E^-$.
\end{lem}
\begin{proof}
Observe that
\begin{align} \label{derive}
\begin{split}
\frac{\partial}{\partial \tau}\gamma_{\eps, z}(\tau, w)&=\Phi'_{\eps}(\tau z^+ + w) z^+,\\
\frac{\partial}{\partial w}\gamma_{\eps, z}(\tau, w) h&=\Phi'_{\eps}(\tau z^+ + w) h \quad \mbox{for any} \, \, h \in E^-.
\end{split}
\end{align}
If $(\tau, w)$ is a critical point of $\gamma_{\eps, z}$ for some $\tau>0$ and $w \in E^-$, then \eqref{derive} implies that
\begin{align} \label{derive1}
\Phi'_{\eps}(\tau z^+ + w) z^+=\Phi'_{\eps}(\tau z^+ + w) h=0 \quad \mbox{for any} \, \, h \in E^-.
\end{align}
This immediately gives that
\begin{align}\label{derive2}
\Phi'_{\eps}(\tau z^+ + w) (\tau z^++w)=\Phi'_{\eps}(\tau z^+ + w) h=0 \quad \mbox{for any} \, \, h \in E^-.
\end{align}
Thus $\tau z^+ + w \in \mathcal{N}$. If $\tau z^+ + w \in \mathcal{N}$ for some $\tau>0$ and $w \in E^-$,  by the definition of $\mathcal{N}$, we then know that \eqref{derive2} necessarily holds. As a consequence, \eqref{derive1} follows. Noting that \eqref{derive}, we then derive that $(\tau, w)$ is a critical point of $\gamma_{\eps, z}$, and the proof is completed.
\end{proof}

\begin{lem} \label{bdd}
For any $z \in E \setminus E^-$, there exist $\eps_{z}>0$ and $R_z>0$ such that, for any $0<\eps<\eps_z$,
\begin{align*}
\Phi_{\eps}(\xi) < 0 \quad \mbox{for any} \,\, \xi \in \hat{E}(z)\setminus B_{R_{z}}(0),
\end{align*}
where $B_{R}(0):=\{z \in E: \|z\| <R\}$.
\end{lem}
\begin{proof}
To prove this, we assume contrarily that there exist $z \in E \setminus E^-$, a sequence $\{\eps_n\} \subset \R^+$ with $\eps_n=o_n(1)$, and a sequence $\{\xi_n\} \subset \hat{E}(z)$ with $\xi_n=\tau_n z^+ + w_n$ for $\{\tau_n\} \subset \R^+$ and $\{w_n\}\subset E^-$ satisfying $\|\xi_n\| \to \infty$ as $ n \to \infty$ such that
\begin{align} \label{conbdd}
\Phi_{\eps_n}(\xi_n) \geq 0.
\end{align}
Define
$$
\bar{\xi}_n:=\frac{\xi_n}{\|\xi_n\|}= \frac{\tau_n}{\|\xi_n\|} z^+ +  \frac{w_n}{\|\xi_n\|}:=\bar{z}_n + \bar{w}_n,
$$
where
$$
\bar{z}_n:= \frac{\tau_n}{\|\xi_n\|}z^+ \in E^+, \quad \bar{w}_n:=\frac{w_n}{\|\xi_n\|} \in E^-.
$$
Therefore,
\begin{align} \label{barxi}
1=\|\bar{\xi}_n\|^2=\|\bar{z}_n\|^2 + \|\bar{w}_n\|^2.
\end{align}
Notice that
\begin{align*}
0 \leq \frac{\Phi_{\eps_n}(\xi_n)}{\|{\xi}_n\|^2}&= \frac 12 \left(\|\bar{z}_n\|^2 - \|\bar{w}_n\|^2\right)
+ \frac 12 \int_{\R}\int_{\R^N} V_{\eps_n}(x)|\bar{z}_n + \bar{w}_n|^2 \, dt dx \\
& \quad -\int_{\R}\int_{\R^N} \frac{F_{\eps_n}(x, |\xi_n|^2)}{\|\xi_n\|^2} \, dtdx \\
& \leq \frac 12 \left(\|\bar{z}_n\|^2 - \|\bar{w}_n\|^2\right)
+ \frac 12 {\|V\|_{\infty}} \left(\|\bar{z}_n\|^2_2 + \|\bar{w}_n\|_2^2\right) \\
&\leq \frac 12 \left(1 +\|V\|_{\infty}\right)\|\bar{z}_n\|^2 - \frac 12 \left(1-\|V\|_{\infty}\right) \|\bar{w}_n\|^2,
\end{align*}
where we used the fact that $F(x, s) \geq 0$ for any $x \in \R^N$ and $s \geq 0$, and Lemma \ref{l2}. This, together with \eqref{barxi}, indicates that
\begin{align*}
\left(\frac{1- \|V\|_{\infty}}{1 + \|V\|_{\infty}}\right) \|\bar{w}_n\|^2 \leq  \|\bar{z}_n\|^2 =1 -\|\bar{w}_n\|^2.
\end{align*}
Hence
\begin{align*}
0 \leq \|\bar{w}_n\|^2 \leq \frac{1 + \|V\|_{\infty}}{2}, \quad  \frac{1 - \|V\|_{\infty}}{2} \leq \|\bar{z}_n\|^2=\frac{\tau_n^2}{\|\xi_n\|^2} \|z^+\|^2\leq 1.
\end{align*}
We now suppose that $\bar{w}_n \rightharpoonup \bar{w}$ and $\bar{z}_n \to \tau z^+$ in $E$ as $n \to \infty$, where
$$
\frac{\tau_n}{\|\xi_n\|} \to \tau \neq 0 \,\, \mbox{in} \,\, \R \, \, \mbox{as} \,\, n \to \infty.
$$
Thus $\bar{\xi}_n  \rightharpoonup \bar{\xi}:= \tau z^+ + \bar{w} \neq 0$ in $E$ as $n \to \infty$. By Lemma \ref{embedding}, it then yields that $\bar{\xi}_n \to \bar{\xi}$ a.e on $\R \times \R^N$ as $n \to \infty$. Setting
$$
\Omega_1:=\left\{(t, x) \in \R \times \R^N: \bar{\xi}(t, x) \neq 0 \right\},
$$
we have that $|\Omega_1| >0$, where $|\Omega| $ denotes the Lebesgue measure of the set $\Omega \subset \R \times \R^N$. Recall that $\|\xi_n\| \to \infty$ as $n \to \infty$, then
\begin{align} \label{nonzero}
{\xi}_n(t, x) \to \infty \,\, \mbox{as}\,\, n \to \infty  \, \, \mbox{for any} \,\, (t, x) \in \Omega_1.
\end{align}
We now apply \eqref{conbdd}, \eqref{barxi}, \eqref{nonzero}, Fatou's lemma, and $(H_3)$ to conclude that
\begin{align*}
0 \leq \limsup_{n \to \infty} \frac{\Phi_{\eps_n}(\xi_n)}{\|{\xi}_n\|^2}& =\frac 12 \limsup_{n \to \infty} \left( \|\bar{z}_n\|^2-\|\bar{w}_n\|^2\right)
+ \frac 12 \limsup_{n \to \infty} \int_{\R}\int_{\R^N} V_{\eps_n}(x) \frac{|\xi_n|^2}{\|\xi_n\|^2} \, dt dx \\
& \quad - \liminf_{n \to \infty} \int_{\R}\int_{\R^N} \frac{F_{\eps_n}(x, |\xi_n|)}{|\xi_n|^2} \frac{|\xi_n|}{\|\xi_n\|^2} \, dtdx \\
& \leq  \frac 12 + \frac {\|V\|_{\infty}}{2} - \int \int_{\Omega_1} \liminf_{n \to \infty} \frac{F_{\eps_n}(x, |\xi_n|)}{|\xi_n|^2} \frac{|\xi_n|}{\|\xi_n\|^2} \, dtdx \\
&=-\infty,
\end{align*}
which is a contradiction. Thus the proof is completed.
\end{proof}

\begin{lem} \label{exist}
For any $z \in E \setminus E^-$ and $0 < \eps< \eps_{z}$, there exist $\tau_{z} >0$ and $w_{z} \in E^-$ such that
$$
\Phi_{\eps}(\tau_{z} z^++ w_{z})=\sup_{\tau \in \R^+, w \in E^-} \Phi_{\eps}(\tau z^++ w),
$$
and $\tau_{z} z^+ + w_{z} \in \mathcal{N}$, where $\eps_{z}>0$ is determined in Lemma \ref{bdd}.
\end{lem}
\begin{proof}
For any $z \in E \setminus E^-$ and $0<\eps<\eps_z$, we define that
$$
\beta_{\eps, z}:=\sup_{\tau \in \R^+, \, w \in E^-} \Phi_{\eps}(\tau z^+ + w).
$$
Obviously, $\beta_{\eps, z} \geq 0$. From Lemma \ref{bdd}, we know that there is a bounded minimizing sequence $\{\xi_n\} \subset \hat{E}(z)$ with $\xi_n= \tau_n z^+ + w_n$ for $\{\tau_n\} \subset \R^+$ and $\{w_n\} \subset E^-$ such that $\Phi_{\eps}(\xi_n)=\beta_{\eps, z}+o_n(1)$. Thus there exist $\tau_{z} \in \R^+$ and $w_{z} \in E^-$ such that $\tau_n \to \tau_z$ in $\R$ and $w_n \rightharpoonup w_z$ in $E$ as $n \to \infty$. Notice that, for any $w, h \in E^-$,
\begin{align} \label{concave}
\begin{split}
\Phi_{\eps}''(w)[h, h]&=-\|h\|^2+ \int_{\R}\int_{\R^N} V_{\eps}(x)|h|^2 \, dtdx- \int_{\R}\int_{\R^N} f_{\eps}(x, |w|) |h|^2 \, dtdx\\
&\quad -\int_{\R} \int_{\R^N} f_{\eps}'(x, |w|) \frac{(w \cdot h)^2}{|w|} \, dtdx \\
& \leq - \left(1- \|V\|_{\infty}\right) \|h\|^2,
\end{split}
\end{align}
where we used the fact that $f(x, s)\geq 0$ for any $x\in\R^N$ and $s \geq 0$, and $(H_4)$. Hence \eqref{concave} suggests that $\Phi_{\eps}$ is strictly concave on $E^-$. Further, we derive that $\Phi_{\eps}$ is weak upper semicontinuous on $E^-$, from which we are able to conclude that $\Phi_{\eps}(\tau_{z} z^+ + w_{z}) =\beta_{\eps, z}$. Observe that, for any $w \in E^-$,
\begin{align*}
\Phi_{\eps}(w)&=-\|w\|^2+ \int_{\R}\int_{\R^N} V_{\eps}(x)|w|^2 \, dtdx- \int_{\R}\int_{\R^N} F_{\eps}(x, |w|)  \, dtdx\\
&\leq - \left(1- \|V\|_{\infty}\right) \|w\|^2
\end{align*}
which shows that $\tau_z >0$. It then follows from Lemma \ref{critical} that $\tau_{z} z^+ + w_{z} \in \mathcal{N}$, and we have finished the proof.
\end{proof}

\begin{lem} \label{compare}
For any $z \in \mathcal{N}$, there holds that
$$
\Phi_{\eps}(\tau z + w) \leq \Phi_{\eps}(z) \quad \mbox{for any} \,\, \tau \in \R^+, w \in E^-.
$$
\end{lem}
\begin{proof}
Since $z \in \mathcal{N}$, then $\Phi_{\eps}'(z)( (\tau^2-1)z + 2\tau w )=0$ for any $\tau \in \R$ and $w \in E^-$. Therefore,
\begin{align}  \nonumber
\Phi_{\eps}(\tau z + w) - \Phi_{\eps}(z)&=\Phi_{\eps}(\tau z + w) - \Phi_{\eps}(z) - \frac 12 \Phi_{\eps}'(z)((\tau^2-1 )z + 2\tau w ) \\ \nonumber
&= -\frac 12  \|w\|^2 + \frac 12 \int_{\R} \int_{\R^N} V_{\eps}(x)|w|^2 \, dtdx - \int_{\R}\int_{R^N} F_{\eps}(x, |\tau z +w|) \, dtdx \\ \label{cmpr}
& \quad + \int_{\R} \int_{\R^N}\frac 12 f_{\eps}(x, |z|) z \cdot ((\tau^2-1 )z + 2\tau w) +  F_{\eps}(x, |z|) \, dtdx \\ \nonumber
& \leq \int_{\R} \int_{\R^N} \frac 12 f_{\eps}(x, |z|) z \cdot ((\tau^2-1 )z + 2\tau w) + F_{\eps}(x, |z|)  \, dtdx \\ \nonumber
& \quad - \int_{\R} \int_{\R^N} F_{\eps}(x, |\tau z +w|) \, dtdx,
\end{align}
where we used the following simple fact,
$$
-\frac 12  \|w\|^2 + \frac 12 \int_{\R} \int_{\R^N} V_{\eps}(x)|w|^2 \, dtdx  \leq -\frac 12  \|w\|^2 + \frac 12 \|V\|_{\infty} \|w\|^2 \leq 0.
$$
For $z, w \in \R^M$, let us now define $h: \R^+ \times \R^N \to \R$ by
\begin{align} \label{defh}
h_{\eps}(\tau, x):= \frac 12 f_{\eps}(x, |z|) z \cdot ((\tau^2-1 )z + 2\tau w) + F_{\eps}(x, |z|) - F_{\eps}(x, |\tau z +w|).
\end{align}
We shall deduce that $h_{\eps}(\tau, x)  \leq 0$ for any $\tau \in \R^+$ and $x \in \R^N$. To do this, we shall consider the following two cases.\smallskip

{\it Case 1:} $z \cdot \left(\tau z + w \right) \leq 0$. \smallskip

Recall that \eqref{notar}, there then holds that
\begin{align} \label{ar}
\frac 12 f_{\eps}(x, s)s^2 - F_{\eps}(x, s) \geq 0 \quad \,\, \mbox{for any} \, \, x\in \R^N, s \geq 0.
\end{align}
Thus, for any $\tau \in \R^+$ and $x \in \R^N$,
\begin{align*}
h_{\eps}(\tau, x) & \leq \frac 12 f_{\eps}(x, |z|) z \cdot \left(\left(\tau^2-1 \right)z + 2\tau w\right) + \frac 12 f_{\eps}(x, |z|)|z|^2  - F_{\eps}(x, |\tau z +w|) \\
& \leq 0,
\end{align*}
where we used the assumption that $z \cdot w \leq - \tau |z|^2$ and the fact that $F_{\eps}(x, s) \geq 0$ for any $x \in \R^N$, $s \geq 0$.

{\it Case 2:} $z \cdot \left(\tau z + w \right) > 0$. \smallskip

Using \eqref{ar}, we can see that $h_{\eps}(0, x) \leq 0$ for any $x \in \R^N$. Moreover,  by $(H_3)$, for any $x \in \R^N$, we have that $h(\tau, x) \to -\infty$ as $\tau \to \infty$.
Note that
\begin{align*}
\partial_{\tau}h_{\eps}(\tau, x)= \left(f_{\eps}(x, |z|)-f_{\eps}(x, |\tau z+w|)\right) z \cdot \left(\tau z + w\right).
\end{align*}
If $\partial_{\tau}h_{\eps}(\tau_0, x)=0$ for some $\tau_0 \in \R^+$, then
\begin{align} \label{stationary}
f_{\eps}(x, |z|)=f_{\eps}(x, |\tau_0 z+w|),
\end{align}
because of $z \cdot \left(\tau z + w \right) > 0$. We now claim that if $f_{\eps}(x, s_1)=f_{\eps}(x, s_2)$ for $s_1, s_2 \in \R^+$, then
$$
F_{\eps}(x, s_1)-F(x, s_2) \leq \frac 12 f_{\eps}(x, s_1)s_1^2-\frac 12 f_{\eps}(x, s_2)s_2^2.
$$
To prove this claim, let us define that $\hat{F}_{\eps}(x, s):=F_{\eps}(x, s)-\frac 12 f_{\eps}(x, s)s^2$. It is easy to see that
$$
\hat{F}_{\eps}'(x, s)=-\frac 12 f_{\eps}'(x, s)s^2 \leq 0,
$$
because $f_{\eps}(x, \cdot)$ is nondecreasing on $\R^+$ for any $x \in \R^N$. Thus $\hat{F}_{\eps}(x, \cdot)$ is nonincreasing on $\R^+$ for any $x \in \R^N$. Consequently, if $s_1 \geq s_2$, then  $\hat{F}_{\eps}(x, s_1) \leq \hat{F}_{\eps}(x, s_2)$, i.e.
$$
F_{\eps}(x, s_1)-F(x, s_2) \leq \frac 12 f_{\eps}(x, s_1)s_1^2-\frac 12 f_{\eps}(x, s_2)s_2^2.
$$
If $s_1 < s_2$, then
\begin{align} \label{decreasing}
F(x, s_1)-F(x, s_2)=- \int_{s_1}^{s_2} f_{\eps}(x, s) s \, ds \leq -\frac 12 f_{\eps}(x, s_1) \left(s_2^2-s_1^2\right).
\end{align}
Since we assumed that $f_{\eps}(x, s_1)=f_{\eps}(x, s_2)$ for $s_1, s_2 \in \R$, then \eqref{decreasing} gives rise to
$$
F(x, s_1)-F(x, s_2) \leq \frac 12 f_{\eps}(x, s_1) s_1^2- \frac 12  f_{\eps}(x, s_2) s_2^2.
$$
Hence the claim follows. Noticing that \eqref{stationary}, we now apply the claim to conclude that
\begin{align*}
h_{\eps}(\tau_0, x)&= \frac 12 f_{\eps}(x, |z|) z \cdot \left(\left(\tau_0^2-1 \right)z + 2\tau_0 w\right) + F_{\eps}(x, |z|) - F_{\eps}(x, |\tau_0 z +w|) \\
& \leq \frac 12 f_{\eps}(x, |z|) z \cdot \left(\left(\tau_0^2-1 \right)z + 2\tau_0 w\right)+ \frac 12 f_{\eps}(x, |z|)|z|^2- \frac 12 f_{\eps}(x, |\tau_0 z + w|) |\tau_0z + w|^2\\
&=\frac 12 f_{\eps}(x, |z|) z \cdot \left(\left(\tau_0^2-1 \right)z + 2\tau_0 w\right)+ \frac 12 f_{\eps}(x, |z|)|z|^2- \frac 12 f_{\eps}(x, |z|) |\tau_0z + w|^2\\
& =- \frac 12  f_{\eps}(x, |z|)|w|^2 \\
& \leq 0.
\end{align*}
Consequently, we obtain that $h_{\eps}(\tau, x) \leq 0$ for any $\tau \in \R^+$ and $x \in \R^N$. Thus, by using \eqref{cmpr}, the lemma then follows, and the proof is completed.
\end{proof}

Letting $P: E \to E^+$ and $Q: E \to E^-$ be orthogonal projections, we introduce another norm on $E$ as
\begin{align} \label{deftopy}
|||z|||:=\max\left\{\|Pz\|, \, \sum_{k=1}^{\infty} \frac{1}{2^{k+1}}|\langle Qz, \,e_k \rangle |\right\} \quad \mbox{for any} \,\, z \in E,
\end{align}
where $\{e_k\} \subset E^-$ is a total orthonormal sequence. The topology generated by $|||\cdot |||$ is denoted by $\mathcal{T}$. Clearly,
\begin{align}\label{norm}
\|Pz\| \leq |||z||| \leq \|z\|.
\end{align}


\begin{lem} \label{deformation}
For any $\eps>0$ small, define
\begin{align} \label{ceps}
c_{\eps}:=\inf_{z \in E \setminus E^-} \inf_{h \in \Gamma(z)} \sup_{z' \in M(z)} \Phi_{\eps}(h(1, z')),
\end{align}
where
\begin{align} \label{defmz}
M(z):=\left\{\tau z + w: \tau \in \R^+, w \in E^-, \, \|\tau z + w \| \leq R_z \right\}
\end{align}
and $R_z>0$ is determined in Lemma \ref{bdd}, in addition,
\begin{align} \label{defGamma}
\Gamma(z):=\left\{h \in C([0, 1] \times M(z)): h \,\, \mbox{satisfies} \,\, (h_1) \mbox{-} (h_4) \right\}
\end{align}
and
\begin{enumerate}
\item [$(h_1)$] $h$ is $\mathcal{T}$-continuous;
\item [$(h_2)$] $h(0, z')=z'$ for any $z' \in M(z)$;
\item [$(h_3)$] $\Phi_{\eps}(z') \geq \Phi_{\eps}(h(t, z'))$ for any $t \in [0, 1], z' \in M(z)$;
\item [$(h_4)$] for every $(t, z') \in [0, 1] \times M(z) $, there is an open neighborhood $W$ in the product topology of $[0, 1]$ and $(E, \mathcal{T})$ such that
$$
\left\{w'-h(s, w'): (s, w') \in W \cap ([0, 1] \times M(z))\right\}
$$
is contained in a finite-dimensional subspace of $E$.
\end{enumerate}
Then there exists a sequence $\{z_n\} \subset E$ such that
\begin{align*}
\Phi_{\eps}(z_n) \leq c_{\eps}+o_n(1), \,\, \,\left(1 + \|z_n\| \right)\Phi_{\eps}'(z_n)=o_n(1).
\end{align*}
\end{lem}
\begin{proof}
For any $\tau >0$, let us first introduce the following notations,
$$
\Phi_{\eps}^{c_{\eps}+ \tau}:=\left\{z \in E: \Phi_{\eps}(z) \leq c_{\eps} + \tau \right\},
$$
and
$$
\Phi_{\eps, \, c_{\eps} - \tau}^{c_{\eps}+ \tau}:=\left\{z \in E: c_{\eps} -\tau  < \Phi_{\eps}(z) \leq c_{\eps} + \tau \right\}.
$$
To prove this lemma, we argue by contradiction that there exists $\tau >0$ such that
\begin{align} \label{con}
\left(1 + \|z\|\right) \|\Phi_{\eps}'(z)\| \geq \tau \quad \mbox{for any} \,\, z \in \Phi_{\eps}^{c_{\eps} + \tau}.
\end{align}
Observe that, for any $z \in \Phi_{\eps}^{c_{\eps}+\tau}$, there exists $\psi_z \in E$ with $\|\psi_z\|=1$ such that
$$
\langle \Phi_{\eps}'(z), \, \psi_{z}\rangle \geq \frac 34 \|\Phi_{\eps}'(z)\|.
$$
This, together with \eqref{con}, leads to
\begin{align} \label{conti1}
\left(1 + \|z\|\right) \langle \Phi_{\eps}'(z), \, \psi_{z}\rangle > \frac{\tau}{2}.
\end{align}
It is simple to check that $\Phi_{\eps}'$ is weakly sequentially continuous on $E$, i.e. if $z_n \rightharpoonup z$ in $E$ as $n \to \infty$, then, for any $\psi \in E$, $\langle \Phi_{\eps}'(z_n), \, \psi \rangle \to \langle \Phi_{\eps}'(z), \, \psi \rangle$ in $\R$ as $n \to \infty$. Moreover, if $z_n \xrightarrow{\mathcal{T}} z$ in $E$ as $n \to \infty$, then $z_n \wto z$ in $E$ as $n \to \infty$. Thus, for any $z \in \Phi_{\eps}^{c_{\eps} + \tau}$, \eqref{conti1} implies that there is a $\mathcal{T}$-open neighborhood $U_z \subset E$ such that, for any $w \in U_z$,
\begin{align}\label{conti2}
\left(1 + \|z\|\right) \langle \Phi_{\eps}'(w), \, \psi_{z}\rangle \geq  \frac{\tau}{2}.
\end{align}
Furthermore, for any $w \in U_z$,
\begin{align} \label{conti3}
\|\left(1 + \|z\|\right) \psi_z\|=1 + \|z\|\leq 2 \left( 1 + \|w\| \right).
\end{align}
We now define that
$$
\mathcal{U}_1:=\left\{U_z: c_{\eps}-\tau < \Phi_{\eps}(z) \leq c_{\eps}+\tau\right\},
\quad \mathcal{U}_2:=\left\{U_z: \Phi_{\eps}(z) \leq c_{\eps}-\tau\right\},
$$
then $\mathcal{U}:=\mathcal{U}_1 \cup \mathcal{U}_2$ forms a $\mathcal{T}$-open covering of $\Phi_{\eps}^{c_{\eps} + \tau}$. Note that $\mathcal{U}$ is metric, hence it is paracompact, which infers that there exists a locally finite $\mathcal{T}$-open covering $\mathcal{M}:=\{M_i: i \in I\}$ of $\Phi_{\eps}^{c_{\eps} + \tau}$, and it is finer than $\mathcal{U}$, where $I$ is an index set. Thus, for any $M_i \in \mathcal{M}$, there is $U_{z_i} \in \mathcal{U}$ for some $z_{i} \in \Phi_{\eps}^{c_{\eps}+ \tau}$ such that $M_i \subset U_{z_i}$. If $U_{z_i} \in \mathcal{U}_1$, we then define that $w_i:=\left(1+ \|z_i\|\right) \psi_{z_i}$. If $U_{z_i} \in \mathcal{U}_2$, we then define that $w_i:=0$. Let $\{\lambda_i: \, i\in \mathcal{I}\}$ be a $\mathcal{T}$-Lipschitz continuous partition of unity subordinated to $\mathcal{M}$, and define
\begin{align*}
\zeta(z):=\sum_{i \in \mathcal{I}} \lambda_i(z) w_i \quad \mbox{for any} \,\, z \in \mathcal{M}.
\end{align*}
Since the covering $\mathcal{M}$ is locally finite, then, for any $z \in \mathcal{M}$, $\zeta(z)<\infty$. In addition, for any $z \in \mathcal{M}$, there is a $\mathcal{T}$-open neighborhood $V_z \subset M_i$ for some $i \in {I}$ such that $\zeta(V_z)$ is contained in a finite-dimension subspace of $E$. Since $\lambda_{i}$ is $\mathcal{T}$-Lipschitz continuous for any $i \in \mathcal{I}$, then there is $L_{z}>0$ such that
\begin{align} \label{lltz}
|||\zeta(z_1)-\zeta(z_2)||| \leq L_z |||z_1-z_2||| \quad \mbox{for any} \,\,z_1, z_2 \in V_z.
\end{align}
By the equivalence of norms in finite-dimensional spaces and \eqref{norm}, it then yields from \eqref{lltz} that
\begin{align} \label{loc}
\|\zeta(z_1)-\zeta(z_2)\| \leq L_z \|z_1-z_2\| \quad \mbox{for any} \,\,z_1, z_2 \in V_z.
\end{align}
Moreover, for any $z \in \mathcal{M}$, \eqref{conti2} and \eqref{conti3} indicate that
\begin{align} \label{negative}
\langle \Phi_{\eps}'(z), \, \zeta(z) \rangle \geq 0
\end{align}
and
\begin{align}\label{glo}
\|\zeta(z)\| \leq 2 \left(1 + \|z\|\right),
\end{align}
respectively. In particular, for any $z \in \Phi^{c_{\eps}+\tau}_{\eps, \, {c_{\eps}-\tau}}$, there holds that
\begin{align} \label{positive}
\langle \Phi_{\eps}'(z), \, \zeta(z)\rangle \geq \frac{\tau}{2}.
\end{align}
Indeed, for any $z  \in \Phi^{c_{\eps}+\tau}_{\eps, {c_{\eps}-\tau}}$, there exist $M_1, \cdots, M_k \in \mathcal{M}$ for some $k \geq 1$ such that $z \in M_i$ for any $1 \leq i \leq k$. Since $z  \in \Phi^{c_{\eps}+\tau}_{\eps, {c_{\eps}-\tau}}$, we then have that $M_i \subset U_{z_i}$ with $z_i \in \Phi^{c_{\eps}+\tau}_{\eps, \, {c_{\eps}-\tau}}$ for any $1 \leq i \leq k$. Thus, from \eqref{conti2},
$$
\langle \Phi_{\eps}'(z), \, \zeta(z)\rangle = \sum_{i=1}^k \lambda_{i}(z) \left(1+ \|z_i\| \right)\langle \Phi_{\eps}'(z), \, \psi_{z_i} \rangle \geq \frac{\tau}{2}.
$$

Let us now consider the Cauchy problem
\begin{align} \label{cauchy}
\left\{
\begin{aligned}
\frac{d}{dt} \eta(t, z) &=-\zeta( \eta(t, z)),\\
\eta(0, z)&=z.
\end{aligned}
\right.
\end{align}
Since $\zeta$ is locally Lipschitz continuous on $\mathcal{M}$, see \eqref{loc}, then, by standard theory of ordinary differential equation in Banach space, $\eta(t, z)$ exists locally in time for any $z \in \mathcal{M}$. Further, by \eqref{glo}, we know that $\eta(t, z)$ exists globally in time for any $z \in \mathcal{M}$. Furthermore, in view of \eqref{negative}, we have that
\begin{align} \label{decreflow}
\frac{d}{dt}\Phi_{\eps}(\eta(t, z))= \langle \Phi_{\eps}'(\eta(t, z)), \, \frac{d}{dt}\eta(t, z)\rangle=-\langle \Phi_{\eps}'(\eta(t ,z)), \, \zeta(\eta(t, z)) \rangle \leq 0.
\end{align}
Choosing $T>4$, we now obtain that
\begin{align} \label{decr}
\eta(T, \Phi_{\eps}^{c_{\eps}+ \tau}) \subset \Phi_{\eps}^{c_{\eps}-\tau}.
\end{align}
In fact, for any $z \in \Phi_{\eps}^{c_{\eps} + \tau}$, if there is $t_0 \in [0, T]$ such that $\eta(t_0, z) \in \Phi_{\eps}^{c_{\eps}-\tau}$, it then follows from \eqref{decreflow} that $\eta(T, z) \in \Phi_{\eps}^{c_{\eps}-\tau}$, and \eqref{decr} follows. Otherwise, there exists $z \in \Phi_{\eps}^{c_{\eps} + \tau}$ such that $\eta(t, z) \in \Phi_{\eps, \, c_{\eps}-\tau}^{c_{\eps} + \tau}$ for any $t \in [0, T]$. According to \eqref{positive}, then
$$
\langle \Phi_{\eps}'(\eta(t, z)),\, \zeta(\eta(t, z))\rangle \geq \frac{\tau}{2}.
$$
Hence
\begin{align*}
\Phi_{\eps}(\eta(T, z))&=\Phi_{\eps}(\eta(0, z)) + \int_{0}^{T} \frac{d}{dt}\Phi_{\eps}(\eta(t ,z)) \, dt \\
&=\Phi_{\eps}(z) - \int_{0}^T \langle \Phi_{\eps}'(\eta(t ,z)),\, \zeta(\eta(t, z))\rangle \, dt \\
& \leq c_{\eps} + \tau - \frac{\tau}{2} T \\
&<c_{\eps} - \tau.
\end{align*}
This is impossible, then \eqref{decr} necessarily holds. In addition, arguing as the proof of \cite[lemma 6.8]{Willem}, we are able to derive that
\begin{enumerate}
  \item [$(i)$] $\eta$ is $\mathcal{T}$-continuous;
  \item [$(ii)$] for any $(t, z) \in [0, T] \times \Phi_{\eps}^{c_{\eps}+\tau}$, there is an open neighborhood $N_{t, z}$ in the product topology of $[0, T]$ and $(E, \mathcal{T})$ such  that
      $$
      \left\{w-\eta(t, w): (t, w) \in N_{t, z}  \cap \left([0, T] \times \Phi_{\eps}^{c_{\eps} + \tau}\right)\right\}
      $$
is contained in a finite-dimensional subspace of $E$.
\end{enumerate}

We now take $z \in E \setminus E^-$ and $h \in \Gamma(z)$ such that
\begin{align} \label{upper}
\sup_{z' \in M(z)} \Phi_{\eps}(h(1, z')) \leq c_{\eps} + \tau.
\end{align}
Define $g : [0 ,1] \times M(z) \to E$ by
\begin{align*}
g(t, z'):=\left\{
\begin{aligned}
&h(2t, z'), \quad &t \in [0, 1/2],\\
&\eta(T(2t-1), h(1, z')), \quad &t \in [1/2, 1],
\end{aligned}
\right.
\end{align*}
and it is easy to check that $g$ enjoys $(h_1)$-$(h_4)$. As a result of \eqref{decr} and \eqref{upper}, we then have that
$$
\Phi_{\eps}(g(1, z')) \leq c_{\eps} -\tau,
$$
which contradicts the definition of $c_{\eps}$. Consequently, there exists a sequence $\{z_n\} \subset E$ so that
$$
\Phi_{\eps}(z_n) \leq c_{\eps}+ o_n(1),\quad \left(1 + \|z_n\|\right) \Phi_{\eps}'(z_n)=o_n(1),
$$
and the proof is completed.
\end{proof}

\begin{lem}\label{sphere}
There exist $r>0$ and $\rho>0$ such that $\Phi_{\eps}{\mid_{S_{r}^+}} \geq \rho$, where
$$
S_{r}^+:=\left\{z \in E^+: \|z\|=r \right\}.
$$
\end{lem}
\begin{proof}
From $(H_1)$ and $(H_2)$, we know that there is $c>0$ such that
$$
G(s) \leq \frac{1-\|V\|_{\infty}}{4} s^2 + c s^{p} \quad \mbox{for any} \,\, s \geq 0.
$$
Thus, by Lemmas \ref{embedding} and \ref{l2},  for any $z \in E^+$,
\begin{align*}
\Phi_{\eps}(z) & =  \frac 12 \|z\|^2+ \frac 12 \int_{\R}\int_{\R^N} V_{\eps}(x)|z|^2 \, dt dx - \int_{\R}\int_{\R^N} F_{\eps}(x, |z|) \, dtdx \\
& \geq \frac 12 \|z\|^2+ \frac 12 \int_{\R}\int_{\R^N} V_{\eps}(x)|z|^2 \, dt dx - \int_{\R}\int_{\R^N} G(|z|) \, dtdx \\
& \geq \frac{1-\|V\|_{\infty}}{4} \|z\|^2-C\|z\|^p,
\end{align*}
from which there exist $r>0$ and $\rho>0$ such that $\Phi_{\eps}{\mid_{S_{r}^+}} \geq \rho$, due to $p>2$.
\end{proof}

\begin{lem} \label{ground}
For any $\eps>0$ small, there holds that
$$
\rho \leq c_{\eps} \leq \inf_{\mathcal{N}} \Phi_{\eps},
$$
where $\rho>0$ is given in Lemma \ref{sphere},  and $c_{\eps}$ is defined by \eqref{ceps}.
\end{lem}
\begin{proof}
We first prove that $c_{\eps} \leq \inf_{\mathcal{N}} \Phi_{\eps}$. For any $z \in \mathcal{N}$, we define that $h: [0, 1] \times M(z) \to E$ by $h(t, z')=z'$. It is simple to check that $h$ satisfies $(h_1)$-$(h_4)$. Thus, by the definition of $c_{\eps}$ and Lemma \ref{compare},
$$
c_{\eps} \leq  \sup_{z' \in M(z)} \Phi_{\eps}(h(1, z'))=\sup_{z' \in M(z)} \Phi_{\eps}(z') \leq \Phi_{\eps}(z),
$$
which implies that $c_{\eps} \leq \inf_{\mathcal{N}} \Phi_{\eps}$. We next show that $c_{\eps} \geq \rho$ for any $\eps>0$ small. To do this,  we suppose by contradiction that $c_{\eps} < \rho$ for some $\eps>0$ small. Therefore, there exist $z \in E \setminus E^-$ and $h \in \Gamma(z)$ such that
\begin{align} \label{contra11}
 \sup_{z' \in M(z)} \Phi_{\eps}(h(1, z')) <  \rho.
\end{align}
Define $H:[0, 1] \times M(z) \to E$ by
$$
H(t, z'):=\left(\|P h(t, z')\|-r\right) \frac{z^+}{\|z^+\|} + Q h(t, z'),
$$
where $r>0$ is given in Lemma \ref{sphere}. Clearly, $H$ fulfills $(h_1)$-$(h_4)$. In addition, $H(t, z')=0$ if and only if $ h(t, z') \in E^+$ and $\|h(t, z')\|=r$. We now claim that $0 \notin H([0, 1] \times \partial M(z))$. To see this, we assume contrarily that there were $(t, z') \in [0, 1] \times \partial M(z)$ such that $H(t, z') =0$, i.e. $ h(t, z') \in E^+$ and $\|h(t, z')\|=r$. It then follows from $(h_3)$ and Lemma \ref{bdd} that
$$
\Phi_{\eps}(h(t, z')) \leq \Phi_{\eps}(z') \leq 0.
$$
However, by Lemma \ref{sphere}, we know that $\Phi_{\eps}(h(t, z')) \geq \rho$. We then reach a contradiction, which in turns indicates that the claim holds. We are now able to apply the homotopy invariance of the degree provided in \cite{KrSz} and $(h_2)$ to conclude that
$$
\mbox{deg}(H(1, \cdot), \,M(z))=\mbox{deg}(H(0, \cdot), \,M(z))=1,
$$
which implies that there exists $\hat{z} \in M(z)$ such that $H(1, \hat{z})=0$. Hence, from Lemma \ref{sphere},
$$
\sup_{z' \in M(z)} \Phi_{\eps}(h(1, z')) \geq \Phi_{\eps}(h(1, \hat{z})) \geq \rho,
$$
which contradicts \eqref{contra11}. Consequently, we have that $c_{\eps} \geq \rho$ for any $\eps>0$ small, and the proof is completed.
\end{proof}

\begin{lem} \label{bounded}
For any $\eps>0$ small,  if $\{z_n\} \subset E$ satisfies that
\begin{align*}
\Phi_{\eps}(z_n) \leq c_{\eps}+o_n(1), \,\, \,\left(1 + \|z_n\| \right)\Phi_{\eps}'(z_n)=o_n(1),
\end{align*}
then $\{z_n\}$ is bounded in $E$.
\end{lem}
\begin{proof}
We argue indirectly that $\{z_n\}$ were unbounded in $E$ and assume that $\|z_n\| \to \infty$ as $n \to \infty$. Define $\xi_{n}:=\frac{z_n}{\|z_n\|}$, and let $\varphi \in C^{\infty}_0(\R^N)$ be such that
\begin{align} \label{defvphi}
\varphi(x):=\left\{
\begin{aligned}
&1,  \quad x \in \overline{(\Lambda^{\delta_0})_{\eps}}, \\
&0,  \quad x \notin N_1(\overline{(\Lambda^{\delta_0})_{\eps}}),
\end{aligned}
\right.
\end{align}
where
\begin{align} \label{defn1}
N_1(\overline{(\Lambda^{\delta_0})_{\eps}}):=\left\{x \in \R^N: \mbox{dist}(x, \, \overline{(\Lambda^{\delta_0})_{\eps}}) <1\right\},
\end{align}
and the constant $\delta_0>0$ is given by \eqref{vloc}. Here the definition of the cutoff function $\varphi$ is inspired by \cite{DX}. Set $\xi_n':=\varphi \xi_n$, then $\{\xi_n'\}$ is bounded in $E$. Moreover, for any $n \in \mathbb{N}^+$, we have that
\begin{align} \label{smalleps}
\|\xi_n'-\xi_n\|=o_{\eps}(1) .
\end{align}
We now claim that there exist $T>0$ and a sequence $\{\tau_n\} \subset \R$ such that
\begin{align} \label{lions1}
\liminf_{n \to \infty} \int_{B(\tau_n, \, T)} \int_{N_1(\overline{(\Lambda^{\delta_0})_{\eps}})} |\xi_n'^+|^2 \,dtdx>0.
\end{align}
To prove this claim, we suppose by contradiction that
\begin{align} \label{vanishing1}
\liminf_{n \to \infty} \sup_{\tau \in \R}\int_{B(\tau, \, T)} \int_{N_1(\overline{(\Lambda^{\delta_0})_{\eps}})} |\xi_n'^+|^2 \,dtdx=0.
\end{align}
By Lions' concentration compactness lemma \cite[Lemma I.1]{Lions}, it then follows from \eqref{vanishing1} that $\xi_n'^+\to 0$ in $L^q$ as $n \to \infty$ for any $2 < q < 2(N+2)/N$. Hence, from \eqref{h1h2}, for any $s \geq 0$,
\begin{align} \label{F}
F_{\eps}(x, s \xi_n'^+)=o_n(1).
\end{align}
Noticing that $\Phi_{\eps}'(z_n)z_n =o_n(1)$ and $\Phi_{\eps}'(z_n) z_n^-=o_n(1)$,  and applying the same arguments as the proof of Lemma \ref{compare}, we can obtain that, for any $s \geq 0$,
$$
\Phi_{\eps}(z_n) \geq \Phi_{\eps}(s \xi_n^+) +o_n(1).
$$
This, jointly with \eqref{smalleps} and \eqref{F}, shows that, for any $n \in \mathbb{N}^+$ large and $\eps>0$ small,
\begin{align} \label{contradiction1}
\begin{split}
c_{\eps} + 2\geq \Phi_{\eps}(z_n) + 1 &\geq \Phi_{\eps}(s \xi_n^+) + \frac 12\geq \Phi_{\eps}(s \xi_n'^+)  + \frac 14\\
& \geq \frac {s^2}{2} \|\xi_n'^+\|^2 + \frac {s^2}{2} \int_{\R}\int_{\R^N}V_{\eps}(x) |\xi_n'^+|^2 \, dtdx + \frac 18 \\
& \geq \frac {s^2}{2} \left(1 - \|V\|_{\infty}\right) \|\xi_n'^+\|^2.
\end{split}
\end{align}
Observe that
\begin{align*}
\frac 12 \left(\|z_n^+\|^2 - \|z_n^-\|^2\right) + \frac 12 \int_{\R} \int_{\R^N} V_{\eps}(x)|z_n|^2 \, dtdx \geq \Phi_{\eps}(z_n)
=\Phi_{\eps}(z_n) - \frac 12 \Phi_{\eps}'(z_n)z_n+ o_n(1),
\end{align*}
where we used the fact that $F(x, s) \geq 0$ for any $x \in \R^N$ and $s \geq 0$. In addition, by \eqref{notar},
\begin{align*}
\Phi_{\eps}(z_n) - \frac 12 \Phi_{\eps}'(z_n)z_n = \int_{\R} \int_{\R^N} \frac 12 f_{\eps}(x, |z_n|)|z_n|^2 - F_{\eps}(x, |z_n|) \, dtdx \geq 0.
\end{align*}
As a result, from two inequalities above and Lemma \ref{l2},
$$
\|z_n^+\|^2 \geq  \left(\frac{1 - \|V\|_{\infty}}{1 + \|V\|_{\infty}} \right)\|z_n^-\|^2 +o_n(1),
$$
which indicates that
$$
\frac{2}{1-\|V\|_{\infty}} \|z_n^+\|^2 \geq \|z_n^+\|^2+ \|z_n^-\|^2 +o_n(1) = \|z_n\|^2 + o_n(1).
$$
Thus
$$
\|\xi_n^+\|^2 \geq \frac {1-\|V\|_{\infty}}{2} + o_n(1).
$$
Consequently, for any $n \in \mathbb{N}^+$ large and $\eps>0$ small, it follows from \eqref{smalleps} that
$$
\|\xi_n'^+\|^2 \geq \frac {1-\|V\|_{\infty}}{4}.
$$
We then reach a contradiction from \eqref{contradiction1} for $s \geq 0$ large enough.
This in turns implies that the claim holds, and we obtain that
\begin{align} \label{lions2}
\liminf_{n \to \infty} \int_{B(\tau_n, \, T)} \int_{N_1(\overline{(\Lambda^{\delta_0})_{\eps}})} |\xi_n|^2 \,dtdx>0,
\end{align}
because of $|\xi_n'| \leq |\xi_n|$. It then yields from Lemma \ref{embedding} that $\bar{\xi}_n(t, x):=\xi_n(t+\tau_n, x) \rightharpoonup \xi \neq 0$ in $E$ as $n \to \infty$. Furthermore, we have that $\bar{\xi}_n \to \xi$ a.e. on $\R \times \R^N$ as $n \to \infty$. Define
$$
\Omega_2:=\left\{(t, x) \in \R \times \R^N: \xi(x, t) \neq 0\right\},
$$
then $\bar{z}_n(t, x):=z_n(t+\tau_n, x) \to \infty$ as $n \to \infty$ for any $(t, x) \in \Omega_2$. Thus, by Fatou's lemma and $(H_3)$,
\begin{align*}
0 \leq \limsup_{n \to \infty} \frac{\Phi_{\eps}(\bar{z}_n)}{\|\bar{z}_n\|^2}& =\frac 12 \limsup_{n \to \infty} \left( \|\bar{\xi}^+_n\|^2-\|\bar{\xi}_n^-\|^2\right)
+ \limsup_{n \to \infty} \int_{\R}\int_{\R^N} V_{\eps}(x) \frac{|\bar{z}_n|^2}{\|\bar{z}_n\|^2} \, dt dx \\
& \quad - \liminf_{n \to \infty} \int_{\R}\int_{\R^N} \frac{F_{\eps}(x, |\bar{z}_n|)}{|\bar{z}_n|^2} \frac{|\bar{z}_n|}{\|\bar{z}_n\|^2} \, dtdx \\
& \leq  \frac 12 + \frac {\|V\|_{\infty}}{2} - \int \int_{\Omega_2} \liminf_{n \to \infty} \frac{F_{\eps}(x, |\bar{z}_n|)}{|\bar{z}_n|^2} \frac{|\bar{z}_n|}{\|\bar{z}_n\|^2} \, dtdx \\
&= -\infty,
\end{align*}
which is impossible. This gives that $\{z_n\}$ is bounded in $E$, and we have completed the proof.
\end{proof}

\begin{lem} \label{existence}
For any $\eps >0$ small, \eqref{equation} admits a ground state $z_{\eps} \in E$.
\end{lem}
\begin{proof}
By Lemma \ref{deformation}, we know that there exists a sequence $\{z_n\} \subset E$ such that
\begin{align*}
\Phi_{\eps}(z_n) \leq c_{\eps}+o_n(1), \,\, \,\left(1 + \|z_n\| \right)\Phi_{\eps}'(z_n)=o_n(1).
\end{align*}
It follows from Lemma \ref{bounded} that $\{z_n\}$ is bounded in $E$. We now set that
$$
{z_n^+}':=\varphi z_n^+,
$$
where $\varphi$ is given by \eqref{defvphi}.  We claim that there exist $T>0$ and a sequence $\{\tau_n\} \subset \R $ such that
\begin{align} \label{nonvanishing}
\liminf_{n \to \infty}\int_{B(\tau_n, \,T)} \int_{N_1(\overline{(\Lambda^{\delta_0})_{\eps}})} |{z_n^+}'|^2 \, dtdx >0,
\end{align}
where $N_1(\overline{(\Lambda^{\delta_0})_{\eps}})$ is given by \eqref{defn1}. Indeed, if the claim were false, then, by Lions' concentration compactness lemma \cite[Lemma I.1]{Lions},
\begin{align} \label{convlp}
{z_n^+}'\to 0 \,\, \mbox{in} \,\,L^q \,\, \mbox{for any} \,\, 2 < q < 2(N+2)/N.
\end{align}
Since $\Phi_{\eps}'(z_n)(z_n^+-z_n^-) =o_n(1)$, then
\begin{align*}
\|z_n\|^2 + \int_{\R}\int_{\R^N} V_{\eps}(x)z_n \cdot (z_n^+ - z_n^-) \,dtdx  &=\int_{\R}\int_{\R^N} f_{\eps}(x, z_n) z_n \cdot (z_n^+ - z_n^-) \, dtdx +o_n(1) \\
&\leq \int_{\R}\int_{\R^N} f_{\eps}(x, z_n) |z_n^+|^2 \, dtdx  +o_n(1).
\end{align*}
This, together with \eqref{deff}, yields  that
\begin{align*}
\|z_n\|^2  -\|V\|_{\infty} \int_{\R}\int_{\R^N}|z_n||z_n^+ - z_n^-| \,dtdx &\leq \int_{\R} \int_{\R^N}   \left(1-\chi(\eps x)\right) g(|z_n|)|z_n^+|^2\,dtdx \\
& \quad + \frac{1-\|V\|_{\infty}}{2} \int_{\R}\int_{\R^N}|z_n^+|^2 \,dtdx + o_n(1).
\end{align*}
By H\"older's inequality and Lemma \ref{l2} and Remark \ref{remlp}, then
\begin{align} \label{l1}
\frac{1-\|V\|_{\infty}}{2} \|z_n\|^2 \leq \int_{\R} \int_{\R^N} \left(1-\chi(\eps x)\right) g(|z_n|)|z_n^+|^2\,dtdx + o_n(1).
\end{align}
From $(H_1)$ and $(H_2)$, we know that there exist $r>0$ and $c_r>0$ such that
$$
g(s) \leq \frac{1-\|V\|_{\infty}}{4} \quad \mbox{for any} \, \, 0 \leq s <r, \quad g(s) \leq c_{r} s^{p-2} \quad \mbox{for any}\,\, s \geq r.
$$
Therefore, by using \eqref{convlp},  H\"older's inequality, and Lemma \ref{bounded}, we conclude from \eqref{l1} that
\begin{align*}
\frac{1-\|V\|_{\infty}}{4} \|z_n\|^2
& \leq \int \int_{\{(t, \, x) \in \R \times \overline{(\Lambda^{\delta_0})_{\eps}}: \, |z_n(t, x)| \geq r\}} g(|z_n|)|z_n^+|^2\,dtdx + o_n(1)\\
&\leq \int \int_{\{(t, \, x) \in \R \times N_1(\overline{(\Lambda^{\delta_0})_{\eps}}): \, |z_n(t, x)| \geq r\}} g(|z_n|)|{z_n^+}'|^2 \,dtdx +o_n(1)\\
&\leq c_{r} \int \int_{\{(t, \, x) \in \R \times N_1(\overline{(\Lambda^{\delta_0})_{\eps}}): \, |z_n(t, x)| \geq r\}} |z_n|^{p-2}|{z_n^+}'|^2 \,dtdx +o_n(1)\\
&\leq c_{r} \|z_n\|_{p}^{p-2} \|{z_n^+}'\|_p^2 +o_n(1)\\
&=o_n(1).
\end{align*}
This indicates that $\|z_n\|=o_n(1)$, then $c_{\eps}=o_n(1)$, which is impossible, see Lemma \ref{ground}. Hence \eqref{nonvanishing} holds, and we have that
\begin{align} \label{non1}
\liminf_{n \to \infty} \int_{B(\tau_n, \, T)} \int_{N_1(\overline{(\Lambda^{\delta_0})_{\eps}})} |z_n^+|^2 \,dtdx>0.
\end{align}
We now define that $\bar{z}_n(t, x):=z_n(t+\tau_n, x)$, then \eqref{non1} implies that $\bar{z}_n^+ \rightharpoonup z_{\eps}^+ \neq 0$ and $\bar{z}_n \rightharpoonup z_{\eps} \neq 0$ in $E$ as $n \to \infty$. By Lemma \ref{embedding}, we get that $\bar{z}_n \to z_{\eps}$ a.e. on $\R \times \R^N$ as $n \to \infty$. In addition, there holds that $\Phi_{\eps}'(z_{\eps})=0$. Consequently, by Fatou's lemma and \eqref{notar},
\begin{align*}
c_{\eps} & \geq  \liminf_{n \to \infty}\left(\Phi_{\eps}(\bar{z}_n) - \frac 12 \Phi_{\eps}'(\bar{z}_n)\bar{z}_n\right) \\
&= \liminf_{n \to \infty} \int_{\R} \int_{\R^N} \frac 12 f_{\eps}(x, |\bar{z}_n|)|\bar{z}_n|^2-F_{\eps}(x, |\bar{z}_n|) \, dtdx \\
& \geq \int_{\R} \int_{\R^N} \frac 12 f_{\eps}(x, |z_{\eps}|)|z_{\eps}|^2-F_{\eps}(x, |z_{\eps}|) \, dtdx \\
& = \Phi_{\eps}(z_{\eps})-\frac 12 \Phi_{\eps}'(z_{\eps})z_{\eps} \\
&= \Phi_{\eps}(z_{\eps}),
\end{align*}
which, along with Lemma \ref{ground}, gives that $c_{\eps}=\inf_{\mathcal{N}} \Phi_{\eps}=\Phi_{\eps}(z_{\eps})$. Hence we have completed the proof.
\end{proof}

\subsection{Exponential decay of ground states}

In what follows, we shall deduce exponential decay of ground states to \eqref{equation}.

\begin{lem} \label{cepsbdd}
For any $\eps >0$ small, there exists $c_0>0$ such that $c_{\eps} \leq c_0$.
\end{lem}
\begin{proof}
For $z_0 \in E \setminus E^-$ given, it follows from Lemmas \ref{exist} and \ref{ground} that, for any $\eps>0$ small,
$$
c_{\eps} \leq \sup_{\tau \in \R^+} \Phi_{\eps} (\tau z_0^+ ).
 $$
In view of Lemma \ref{bdd}, for any $\eps>0$ small, we deduce that there exists $\tau_0>0$ such that $\Phi_{\eps}(\tau z_0^+) \leq 0$ for any $\tau \geq \tau_0$, which then shows that
\begin{align} \label{bdd1}
c_{\eps} \leq \sup_{\tau \in [0, \tau_0]} \Phi_{\eps}(\tau z_0^+).
\end{align}
Thus, for any $\eps>0$ small, it yields from \eqref{bdd1} that $c_{\eps} \leq c_0$, and the proof is completed.
\end{proof}

\begin{lem} \label{unibdd}
Let $z_{\eps}$ be a ground state to \eqref{equation}, then there exist $c_1$, $c_2>0$ such that
$$
c_1 \leq \|z_{\eps}\| \leq c_2.
$$
\end{lem}
\begin{proof}
Since, for any $\eps >0$ small, $c_{\eps} \geq \rho$,  see Lemma \ref{deformation}, then there exists $c_1>0$ such that $\|z_{\eps}\| \geq c_1$. Otherwise, we have that $c_{\eps}=o_{\eps}(1)$, which is impossible. On the other hand,  for any $\eps>0$ small, Lemma \ref{cepsbdd} indicates that $\Phi_{\eps}(z_{\eps})=c_{\eps} \leq c_0$. In addition, we know that $\Phi_{\eps}'(z_{\eps})z_{\eps}=0$, because $z_{\eps}$ is a ground state to \eqref{equation}. Thus, arguing as the proof of Lemma \ref{bounded},  we are able to prove that there exists $c_2>0$ such that $\|z_{\eps}\| \leq c_2$. Hence the proof is completed.
\end{proof}

\begin{lem} \label{bddbr}
Let $z_{\eps}$ be a ground state to \eqref{equation}, then $z_{\eps} \in B^q$, and
$$
\|z_{\eps}\|_{B^q} \leq C \quad \mbox{for any} \,\, q \geq 2,
$$
where the Banach space $B^q$ is defined by \eqref{defbq}.
\end{lem}
\begin{proof}
This lemma can be proved by using Lemmas \ref{regularity}-\ref{est2.7}, and the iteration technique shown in the proof of \cite[Lemma A.5]{DX}.
\end{proof}

\begin{lem} \label{sl}
Let $z_{\eps}$ be a ground state to \eqref{equation}, then there exist a number $m \in \mathbb{N}^+$, $m$ nontrivial functions $z_1, \cdots, z_m \in E$, and $m$ sequences $\{(\tau_{\eps, 1}, \, y_{\eps, 1})\}, \cdots,\{(\tau_{\eps, m}, \, y_{\eps, m})\} \subset \R \times \R^N$ such that, up to subsequences if necessary,
\begin{enumerate}
\item [$(i)$] $\eps y_{\eps, k} \to y_{k} \in \Lambda^{\delta_0}  \,\, \mbox{in} \,\,\R^N \,\,\mbox{as}  \,\, \eps \to 0^+ \,\, \mbox{for any} \,\, 1 \leq k \leq m$ and $\left|\tau_{\eps, k_1} - \tau_{\eps, k_2}\right| \to \infty$ or $\left|y_{\eps, k_1} -y_{\eps, k_2}\right| \to \infty$ for any $1\leq k_1 \neq k_2 \leq m$, where $\delta_0>0$ is given by \eqref{vloc};
\item [$(ii)$] there holds that
\begin{align} \label{bl}
z_{\eps}- \sum_{k=1}^{m}z_k(\cdot-\tau_{\eps, k}, \, \cdot-y_{\eps, k}) =o_{\eps}(1)  \,\, \mbox{in} \,\, E,
\end{align}
where, for any $1 \leq k \leq m$, $z_k$ is a nontrivial solution to the system
$$
Lz + V(y_{k})z=f(y_k, |z|)z.
$$
\end{enumerate}
\end{lem}
\begin{proof}
We first claim that
\begin{align} \label{non2}
\liminf_{\eps \to 0^+} \sup_{(\tau, \,y) \in \R \times \R^N} \int_{B(\tau, \, T)} \int_{B(y, \, R)} |z_{\eps}|^2 \, dtdx >0.
\end{align}
Indeed, if \eqref{non2} were false,  then, by Lemma \ref{concentration}, we get that $z_{\eps} \to 0$ in $L^p$ as $\eps \to 0^+$ for any $2 <p<2(N+2)/N$. Note that
\begin{align*}
c_{\eps}&=\Phi_{\eps}(z_{\eps})- \frac 12 \Phi_{\eps}'(z_{\eps})z_{\eps} = \int_{\R} \int_{\R^N} \frac 12 f_{\eps}(x, |z_{\eps}|)|z_{\eps}|^2 -F_{\eps}(x, |z_{\eps}|) \, dtdx.
\end{align*}
As a consequence of \eqref{h1h21} and \eqref{h1h2}, we then obtain that $c_{\eps}=o_{\eps}(1)$, which is impossible, see Lemma \ref{ground}. Hence the claim holds, and we know that there exists a sequence $\{(\tau_{\eps, 1}, \, y_{\eps, 1})\} \subset \R \times \R^N$ such that
\begin{align} \label{nontrivial1}
\liminf_{\eps \to 0^+} \int_{B(\tau_{\eps, 1}, \, T)} \int_{B(y_{\eps, 1}, \, R)} |z_{\eps}|^2 \, dtdx >0.
\end{align}
Define
$$
\bar{z}_{\eps}(t, x):=z_{\eps}(t+ \tau_{\eps, 1}, x+y_{\eps, 1}),
$$
it then follows from \eqref{nontrivial1} and Lemma \ref{embedding} that $\bar{z}_{\eps} \wto z_1 \neq 0$ in $E$ as $n \to \infty$. Since $z_{\eps}$ is a ground state to \eqref{equation}, then
\begin{align} \label{zeps}
L\bar{z}_{\eps} + V_{\eps}(x + y_{\eps, 1})\bar{z}_{\eps}=f_{\eps}(x+y_{\eps, 1}, |\bar{z}_{\eps}|)\bar{z}_{\eps}.
\end{align}
We now deduce that $\eps y_{\eps, 1} \to y_{1} \in \Lambda^{\delta_0}$ in $\R^N$ as $\eps \to 0^+$. To do this, let us first prove that $\{\eps y_{\eps, 1}\} \subset \R^N$ is bounded. We assume contrarily that $|\eps y_{\eps, 1}| \to \infty$ in $\R$ as $\eps \to 0^+$. Thus, from \eqref{zeps}, we have that
\begin{align} \label{z1}
Lz_1 + V_1z_1=\tilde{g}(|z_1|)z_1,
\end{align}
where $V_1:=\lim_{\eps \to 0^+} V_{\eps}(x + y_{\eps, 1})$, and $\tilde{g}$ is defined by \eqref{tildeg}. By taking the scalar product to \eqref{z1} with $z_1^+-z_1^-$ and integrating on $\R \times \R^N$, then
\begin{align} \label{z1neq0}
\begin{split}
0&=\|z_1\|^2+V_1 \int_{\R}\int_{\R^N}z_1 \cdot \left(z_1^+- z_1^-\right) \, dtdx -\int_{\R}\int_{\R^N} \tilde{g}(|z_1|)z_1 \cdot \left(z_1^+-z_1^-\right) \, dtdx \\
&\geq \|z_1\|^2-\|V\|_{\infty} \|z_1\|^2-\frac{1-\|V\|_{\infty}}{2}\|z_1\|^2\\
&=\frac{1-\|V\|_{\infty}}{2}\|z_1\|^2,
\end{split}
\end{align}
where we used H\"older's inequality and Lemma \ref{l2}.  As a result of \eqref{z1neq0}, we then obtain that $z_1=0$, which is a contradiction. Thus we know that $\{\eps y_{\eps, 1}\}$ is bounded in $\R^N$. We now suppose that $\eps y_{\eps, 1} \to y_1$ in $\R^N$ as $\eps \to 0^+$. If $y_1 \notin \Lambda^{\delta_0}$,  we conclude from \eqref{zeps} that
\begin{align} \label{blcon}
Lz_1 + \tilde{V}_1z_1=\tilde{g}(|z_1|)z_1,
\end{align}
where $\tilde{V}_1:=\lim_{\eps \to 0^+} V_{\eps}(x + y_{\eps, 1})$. By \eqref{blcon}, we are able to reach a contradiction as before. Accordingly, $\eps y_{\eps, 1} \to y_{1} \in \Lambda^{\delta_0}$ in $\R^N$ as $\eps \to 0^+$. It then follows from \eqref{zeps} that
\begin{align} \label{z11}
Lz_1 + V(y_1) z_1=f(y_1, |z_1|)z_1.
\end{align}
Taking the scalar product to \eqref{z11} with $z_1^+-z_1^-$ and integratimg on $\R \times \R^N$, we find that
\begin{align} \label{lowerbd}
\begin{split}
\|z_1\|^2+V(y_1) \int_{\R}\int_{\R^N}z_1 \cdot \left(z_1^+- z_1^-\right) \, dtdx &=\int_{\R}\int_{\R^N} f(y_1, |z_1|)z_1 \cdot \left(z_1^+-z_1^-\right) \, dtdx \\
& \leq \frac{1-\|V\|_{\infty}}{2}\|z_1\|^2 + c\|z_1\|^p,
\end{split}
\end{align}
where we used the inequality \eqref{h1h21} with $\gamma=\frac{1-\|V\|_{\infty}}{2}$, H\"older's inequality, and Lemmas \ref{embedding}-\ref{lp}. Notice that
$$
\left|V(y_1) \int_{\R}\int_{\R^N}z_1 \cdot \left(z_1^+- z_1^-\right) \, dtdx \right| \leq \|V\|_{\infty} \|z_1\|^2,
$$
then \eqref{lowerbd} leads to
$$
\frac{1-\|V\|_{\infty}}{2} \|z_1\|^2 \leq c\|z_1\|^p,
$$
from which we derive that there exists $c_p>0$ such that $\|z_1\| \geq c_p$.

We now define that
$$
z_{\eps, 1}(t, x):=z_{\eps}(t, x)-z_1(t-\tau_{\eps, 1}, x-y_{\eps, 1}).
$$
If $\|z_{\eps, 1}\|=o_{\eps}(1)$, then the proof is completed. Otherwise, there holds that $\lim_{\eps \to 0^+} \|z_{\eps, 1}\| >0$. Since $\bar{z}_{\eps} \wto z_1$ in $E$ as $n \to \infty$, then
\begin{align} \label{lb1}
\|z_{\eps, 1}\|^2=\|z_{\eps}\|^2 -\|z_1\|^2 +o_{\eps}(1).
\end{align}
Noting that \eqref{zeps} and \eqref{z11}, by standard arguments, we get that
\begin{align}\label{zeps1}
Lz_{\eps, 1}+V_{\eps}(x)z_{\eps, 1}=f_{\eps}(x, |z_{\eps, 1}|)z_{\eps, 1} +o_{\eps}(1).
\end{align}
Taking the scalar product to \eqref{zeps1} with $z_{\eps,1}^+-z_{\eps,1}^-$ and integrating on $\R \times \R^N$, we conclude that
$$
\|z_{\eps,1}\|^2+\int_{\R}\int_{\R^N}V_{\eps}(x)z_{\eps,1} \cdot (z_{\eps,1}^+- z_{\eps,1}^-) \, dtdx =\int_{\R}\int_{\R^N} f_{\eps}(x, |z_{\eps, 1}|)z_{\eps, 1} \cdot (z_{\eps,1}^+-z_{\eps,1}^-) \, dtdx + o_{\eps}(1).
$$
Similarly, by using \eqref{h1h21}, H\"older inequality, and Lemmas \ref{l2}-\ref{lp}, we can deduce that
\begin{align} \label{con2}
\frac{1-\|V\|_{\infty}}{2} \|z_{\eps, 1}\|^2 \leq c \|z_{\eps, 1}\|^{p}_p+ o_{\eps}(1).
\end{align}
Recall that $\lim_{\eps \to 0^+} \|z_{\eps, 1}\| >0$, it then follows from \eqref{con2} and Lemma \ref{concentration} that
$$
\liminf_{\eps \to 0^+} \sup_{(\tau, \,y) \in \R \times \R^N} \int_{B(\tau, \, T)} \int_{B(y, \, R)} |z_{\eps, 1}|^2 \, dtdx >0.
$$
Thus there exists a sequence $\{(\tau_{\eps, 2}, \, y_{\eps, 2})\} \subset \R \times \R^N$ such that
\begin{align} \label{nonzero1}
\liminf_{\eps \to 0^+} \int_{B(\tau_{\eps, 2}, \, T)} \int_{B(y_{\eps, 2}, \, R)} |z_{\eps, 1}|^2 \, dtdx >0,
\end{align}
from which we know that
\begin{align} \label{concentra}
\liminf_{\eps \to 0^+} \int_{B(\tau_{\eps, 2}-\tau_{\eps, 1}, \, T)} \int_{B(y_{\eps, 2}-y_{\eps, 1}, \, R)} |z_{\eps, 1}(t + \tau_{\eps, 1}, x + y_{\eps, 1})|^2 \, dtdx >0.
\end{align}
Since $z_{\eps, 1}(\cdot + \tau_{\eps, 1}, \cdot + y_{\eps, 1}) \wto 0$ in $E$ as $\eps \to 0^+$, then \eqref{concentra} and Lemma \ref{embedding} yields that
$$
\left|\tau_{\eps, 1}-\tau_{\eps, 2}\right| \to \infty \,\, \mbox{or} \, \, \left|y_{\eps, 1}-y_{\eps, 2}\right| \to \infty \quad \mbox{as} \, \, \eps \to 0^+.
$$
Define
$$
\bar{z}_{\eps, 1}(t, x):=z_{\eps, 1}(t+ \tau_{\eps, 2}, x+y_{\eps, 2}).
$$
It then follows from \eqref{nonzero1} and Lemma \ref{embedding} that $\bar{z}_{\eps, 1} \wto z_2 \neq 0$ in $E$ as $n \to \infty$. In addition,  from \eqref{zeps1}, we obtain that
$$
L\bar{z}_{\eps, 1}+V_{\eps}(x+ y_{\eps, 2}) \bar{z}_{\eps, 1}=f_{\eps}(x, |\bar{z}_{\eps, 1}|)\bar{z}_{\eps, 1} +o_{\eps}(1).
$$
By a similar way, we can deduce that $\eps y_{\eps, 2} \to y_2 \in \Lambda^{\delta_0}$ in $\R^N$ as $\eps \to 0^+$, and
$$
Lz_{2} +V(y_2)z_2 = f(y_2, |z_2|)z_2.
$$
Furthermore, $\|z_2\| \geq c_p$.

We now define that
$$
z_{\eps, 2}(t, x):=z_{\eps, 1}-z_2(t -\tau_{\eps, 2}, x-y_{\eps, 2}).
$$
If $\|z_{\eps, 2}\|=o_{\eps}(1)$, then the proof is done. Otherwise, we have that $\lim_{\eps \to 0^+} \|z_{\eps, 2}\| >0$. Since $\bar{z}_{\eps, 1} \wto z_2$ in $E$ as $n \to \infty$, then
\begin{align*}
\|z_{\eps, 2}\|^2=\|z_{\eps, 1}\|^2-\|z_2\|^2+o_{\eps}(1).
\end{align*}
This, along with \eqref{lb1}, indicates that
$$
\|z_{\eps, 2}\|^2=\|z_{\eps}\|^2-\|z_1\|^2-\|z_2\|^2+o_{\eps}(1).
$$
Applying the same arguments as before, we can derive that there exists a sequence $\{(\tau_{\eps, 3}, \, y_{\eps, 3})\} \subset \R \times \R^N$ such that $\eps y_{\eps, 3} \to y_3 \in \Lambda^{\delta_0}$ in $\R^N$ as $\eps \to 0^+$, and for any $1 \leq k_1 \neq k_2 \leq 3$,
$$
\left|\tau_{\eps, k_1} - \tau_{\eps, k_2}\right| \to \infty \, \, \mbox{or} \,\, \left|y_{\eps, k_1}-y_{\eps, k_2}\right| \to \infty \quad \mbox{as} \,\, \eps \to 0^+.
$$
Define
$$
\bar{z}_{\eps, 2}(t, x):=z_{\eps, 2}(t+ \tau_{\eps, 3}, x+ y_{\eps, 3}),
$$
then $ \bar{z}_{\eps, 2} \wto z_3 \neq 0$ in $E$ as $\eps \to 0^+$, and
$$
Lz_{3} +V(y_3)z_3 = f(y_3, |z_3|)z_3.
$$
Furthermore, $\|z_3\| \geq c_p$.

By iterating $m$ times, we are able to obtain $m$ sequences $\{(\tau_{\eps, 1},\, y_{\eps, 1})\}, \cdots, \{(\tau_{\eps, m},\, y_{\eps, m})\} \subset \R \times \R^N$ such that $\eps y_{\eps, k} \to y_k \in \Lambda^{\delta_0}$ in $\R^N$ as $\eps \to 0^+$ for any $1 \leq k \leq m$ and
$$
\left|\tau_{\eps, k_1} - \tau_{\eps, k_2}\right| \to \infty \, \, \mbox{or} \,\, \left|y_{\eps, k_1}-y_{\eps, k_2}\right| \to \infty \,\, \mbox{as} \,\, \eps \to 0^+ \quad \mbox{for any} \,\,1 \leq k_1 \neq k_2 \leq m.
$$
There also exist $m$ nontrivial functions $z_1, \cdots, z_m \in E$ such that, for any $1 \leq k \leq m$, $\|z_k\| \geq c_p$ and
$$
Lz_k +V(y_k)z_k = f(y_k, |z_k|)z_k.
$$
In addition,
$$
0 \leq \|z_{\eps}\|^2-\sum_{k=1}^{m} \|z_k\|^2+o_{\eps}(1).
$$
Since, for any $1 \leq k \leq m$, $\|z_k\| \geq c_p$, and $\|z_{\eps}\| \leq c_2$, see Lemma \ref{unibdd}, then the procedure has to terminate at some finite index $m$ with $\|z_{\eps, m}\|=o_{\eps}(1)$, and the proof is completed.
\end{proof}

Let $\{\eps_n\} \subset \R^+$ be such that $\eps_n=o_n(1)$, and assume that $\lim_{n \to \infty} \eps_n y_{\eps_n ,k}$ exists for any $1 \leq k \leq m$. We write
$$
\left\{x_1, x_2, \cdots, x_{\tilde{m}}\right\}:=\left\{\lim_{n \to \infty} \eps_n y_{\eps_n ,k}: k=1,2, \cdots, m\right\},
$$
where $1\leq \tilde{m} \leq m$, and $x_{k_1} \neq x_{k_2}$ for any $1\leq k_1 \neq k_2 \leq \tilde{m}$. Define
\begin{align*}
\nu:=\left\{
\begin{aligned}
&\frac {1}{10} \mbox{min} \left\{|x_{k_1}-x_{k_2}|: 1\leq k_1 \neq k_2 \leq \tilde{m}\right\}, & \tilde{m} \geq 2,\\
&\infty, &\tilde{m}=1.
\end{aligned}
\right.
\end{align*}

\begin{lem} \label{decay}
Let $0<\delta< \nu$, then there exist $c>0$ and $C>0$ such that, for any $n \in \mathbb{N}^+$ large,
$$
\int_{\R}\int_{\mathcal{D}_{n, k}}  |\nabla z_{\eps_n}|^2 + |z_{\eps_n}|^2 \, dtdx \leq  C \, \textnormal{exp}\left(-c \, \eps_n^{-1}\right),
$$
\end{lem}
where $1 \leq k \leq m$, and
$$
\mathcal{D}_{n, k}:=\overline{B(y_{\eps_n, k}, \, \delta \eps_n^{-1}+2)} \setminus B(y_{\eps_n, k}, \, \delta \eps_n^{-1} -2).
$$
\begin{proof}
To prove this, we shall make use of the iteration technique developed in \cite{ChWa}. Let us define that
$$
A_{n, k}:=\overline{B(y_{\eps_n, k}, \, \frac 32\delta \eps_n^{-1})}\setminus B(y_{\eps, k}, \, \frac 12 \delta \eps_n^{-1}).
$$
By the definition of $\nu$, then, for any $0<\delta <\nu$,
\begin{align} \label{distlimit}
\mbox{dist}\left(y_{\eps_n, k'}, \, A_{n, k}\right) \to \infty \,\, \mbox{as} \,\, n \to \infty \quad \mbox{for any} \,\,1 \leq k', k \leq m.
\end{align}
Notice that, for any $1 \leq k \leq m$,
\begin{align} \label{inftylimit}
\lim_{R \to \infty}\int_{\R}\int_{\R^N \setminus B(y_{\eps_n, k}, \, R)} |z_k(\cdot,\, \cdot-y_{\eps_n, k})|^2 \, dtdx =o_n(1).
\end{align}
Define
$$
\mathcal{N}_{n, k}:=\left\{x \in \R^N: \textnormal{dist}(x, \, A_{n, k}) \leq 1\right\}.
$$
From \eqref{bl}, \eqref{distlimit} and \eqref{inftylimit}, we then have that
\begin{align} \nonumber
\int_{\R} \int_{\mathcal{N}_{n, k}} |z_{\eps_n}|^2 \, dtdx
&= \int_{\R}\int_{\mathcal{N}_{n, k}} |z_{\eps_n}- \sum_{k=1}^m z_k(\cdot-\tau_{\eps_n, k}, \cdot-y_{\eps_n, k}) + \sum_{k=1}^m z_k(\cdot -\tau_{\eps_n, k}, \cdot-y_{\eps_n, k}) |^2 \, dtdx \\ \nonumber
& \leq  2\int_{\R}\int_{\mathcal{N}_{n, k}}|z_{\eps_n}- \sum_{k=1}^m z_k(\cdot-\tau_{\eps_n, k}, \cdot-y_{\eps_n, k})|^2 \, dtdx \\ \label{estii}
& \quad + 2\int_{\R}\int_{\mathcal{N}_{n, k}}|\sum_{k=1}^m z_k(\cdot-\tau_{\eps_n, k}, \cdot-y_{\eps_n, k}) |^2 \, dtdx \\ \nonumber
& = 2\int_{\R}\int_{\mathcal{N}_{n, k}}|z_{\eps_n}- \sum_{k=1}^m z_k(\cdot-\tau_{\eps_n, k}, \cdot-y_{\eps_n, k})|^2 \, dtdx \\ \nonumber
& \quad + 2\int_{\R}\int_{\mathcal{N}_{n, k}}|\sum_{k=1}^m z_k(\cdot, \cdot-y_{\eps_n, k}) |^2 \, dtdx \\ \nonumber
&=o_n(1).
\end{align}
According to Lemma \ref{bddbr}, for any $n \in \mathbb{N}$ large, we know that $\|z_{\eps_n}\|_{B^q} \leq C$ for any $q \geq 2$. By H\"older's inequality, we then get from \eqref{estii} that
\begin{align} \label{esti00}
\int_{\R} \int_{\mathcal{N}_{n, k}} |z_{\eps_n}|^q \, dtdx =o_n(1) \quad \mbox{for any} \,\, q >2.
\end{align}
Define
$$
\hat{z}_{\eps_n}(t, x):=\left(u_{\eps_n}(t, x), v_{\eps_n}(-t, x)\right).
$$
Since, for any $n \in \mathbb{N}^+$ large, $z_{\eps_n}$ is a ground state to \eqref{equation}, then
\begin{align}\label{hatzepsn}
\partial_{t} \hat{z}_{\eps_n}-\Delta \hat{z}_{\eps_n} + \hat{z}_{\eps_n}=h,
\end{align}
where $h:=(h_1, h_2)$ with
$$
h_1(t, x):=-V_{\eps_n}(x)  v_{\eps_n}(t, x)+f_{\eps_n}(x, |z_{\eps_n}(t ,x)|)v_{\eps_n}(t, x),
$$
and
$$
h_2(t, x):=-V_{\eps_n}(x) u_{\eps_n}(-t, x) + f_{\eps_n}(x, |z_{\eps_n}(-t, x)|)u_{\eps_n}(-t, x).
$$
It then follows from Corollary \ref{estimate} and \eqref{esti00} that, for any $\gamma>0$, there exists $N \in \mathbb{N}^+$ such that, for any $n \geq N$,
\begin{align}\label{small}
|\hat{z}_{\eps_n}(t, x)| \leq \gamma \quad \mbox{for any}\,\, t \in \R, x \in A_{n, k}.
\end{align}

For any ${l} \in \mathbb{N}^+$,  we now define that
$$
\mathcal{A}_{n, l}:=\overline{B(y_{\eps_n, k}, \frac 32\delta \eps_n^{-1}-l)} \setminus B(y_{\eps_n, k}, \, \frac 12 \delta \eps_n^{-1} +l).
$$
Let $\zeta_l \in C^{\infty}(\R, [0, 1])$ be a cut-off function with $|\zeta_l '(\tau)| \leq 4$ for any $\tau \in \R$, and
\begin{align*}
\zeta_{n, l}(\tau):=\left\{
\begin{aligned}
&0, \quad \tau \leq  \frac 12 \delta \eps_n^{-1} +l-1 \,\, \mbox{or} \,\, \tau \geq \frac 32 \delta \eps_n^{-1}-l +1, \\
&1, \quad \frac 12 \delta \eps_n^{-1}+l \leq \tau \leq \frac 32 \delta \eps_n^{-1}-l.
\end{aligned}
\right.
\end{align*}
For any $x \in \R^N$, we define that $\psi_{n, l}(x):=\zeta_{n, l}(|x-y_{\eps_n, k}|)$. Taking the scalar product to \eqref{hatzepsn} with $\psi_{n, l}^2 \hat{z}_{\eps_n}$ and integrating on $\R \times \R^N$, we obtain that
\begin{align} \label{integrate}
\begin{split}
&\int_{\R}\int_{\R^N} \partial_t \hat{z}_{\eps_n} \cdot \hat{z}_{\eps_n} \psi_{n, l}^2 \, dtdx -\int_{\R} \int_{\R^N} \Delta \hat{z}_{\eps_n} \cdot \hat{z}_{\eps_n}\psi_{n, l}^2  \, dtdx + \int_{\R} \int_{\R^N} |\hat{z}_{\eps_n}|^2 \psi_{n, l}^2 \, dtdx \\
&= \int_{\R} \int_{\R^N} h \cdot \hat{z}_{\eps_n} \psi_{n, l}^2 \, dtdx.
\end{split}
\end{align}
Note that
$$
\int_{\R}\int_{\R^N} \partial_t \hat{z}_{\eps_n} \cdot \hat{z}_{\eps_n} \psi_{n, l}^2 \, dtdx =\frac 12 \int_{\R} \partial_t \int_{\mathcal \R^N} |\hat{z}_{\eps_n}|^2 \psi_{n, l}^2 \, dxdt =0,
$$
and
$$
-\int_{\R} \int_{\R^N} \Delta \hat{z}_{\eps_n} \cdot \hat{z}_{\eps_n}\psi_{n, l}^2  \, dtdx = \int_{\R} \int_{\mathcal{A}_{n, l-1}} |\nabla \hat{z}_{\eps_n}|^2 \psi_{n, l}^2 \, dtdx + 2\int_{\R} \int_{\mathcal{A}_{n, l-1}} \left(\nabla \hat{z}_{\eps_n} \cdot \nabla \psi_{n, l}\right) \cdot\left(\hat{z}_{\eps} \psi_{n, l}\right)\, dtdx.
$$
Since, for any $l \in \mathbb{N}^+$, $\mathcal{A}_{n ,l} \subset A_{n, k}$, and $\|V\|_{\infty} <1$, it then follows from \eqref{small} that there exists $0<\beta<1$ such that, for any $n \in \mathbb{N}^+$ large,
$$
\int_{\R}\int_{\mathcal{A}_{n, l-1}} h\cdot \hat{z}_{\eps_n} \psi_{n, l}^2 \, dtdx \leq \beta \int_{\R} \int_{\mathcal{A}_{n, l-1}} |\hat{z}_{\eps_n}|^2 \psi_{n, l}^2 \, dtdx.
$$
Thus \eqref{integrate} implies that
\begin{align*}
\int_{\R} \int_{\mathcal{A}_{n, l-1}} |\nabla \hat{z}_{\eps_n}|^2 \psi_{n, l}^2 + \left(1-\beta \right) |\hat{z}_{\eps_n}|^2 \psi_{n, l}^2 \, dtdx &\leq -2\int_{\R} \int_{\mathcal{A}_{n, l-1}} \left(\nabla \hat{z}_{\eps_n} \cdot \nabla \psi_{n, l}\right) \cdot\left(\hat{z}_{\eps} \psi_{n, l}\right)\, dtdx \\
&\leq \hat{C} \int_{\R} \int_{\mathcal{A}_{n, l-1}\setminus \mathcal{A}_{n, l}} |\nabla \hat{z}_{\eps_n}|  |\hat{z}_{\eps}|\, dtdx.
\end{align*}
Observe that $\mathcal{A}_{n, l} \subset \mathcal{A}_{n, l-1}$, then there is $\hat{c}>0$ such that
$$
\int_{\R} \int_{\mathcal{A}_{n, l}} |\nabla \hat{z}_{\eps_n}|^2  + |\hat{z}_{\eps_n}|^2 \, dtdx
\leq   \hat{c} \int_{\R} \int_{\mathcal{A}_{n, l-1}\setminus \mathcal{A}_{n, l}} |\nabla \hat{z}_{\eps_n}|^2 + |\hat{z}_{\eps}|^2  \, dtdx.
$$
This gives that $a_{n, l} \leq \hat{c} \left(a_{n, l-1}-a_{n, l} \right)$, where
$$
a_{n, l}:=\int_{\R} \int_{\mathcal{A}_{n, l}} |\nabla \hat{z}_{\eps_n}|^2  + |\hat{z}_{\eps_n}|^2  \, dtdx.
$$
Hence  $a_{n, l} \leq \theta a_{n, l-1}$ for $\theta:=\frac{\hat{c}}{\hat{c}+1}<1$, from which we get that $a_l \leq \theta^l a_0$, where
$$
a_{n, 0}:=\int_{\R} \int_{{A}_{n, k}} |\nabla \hat{z}_{\eps_n}|^2  + |\hat{z}_{\eps_n}|^2  \, dtdx.
$$
Recall that $\{z_{\eps_n}\}$ is bounded in $E$, see Lemma \ref{unibdd}, then $a_l \leq \bar{c}\,\theta^l=\bar{c}\,e^{l \ln{\theta}}$ for some $\bar{c}>0$. Taking $l=[\frac 12 \delta \eps_n^{-1}]-2$, and letting $n \in \mathbb{N}^+$ large if necessary such that
$$
\left[\frac 12 \delta \eps_n^{-1}\right]-2 \geq \frac 14 \delta \eps_n^{-1},
$$
we then obtain that
\begin{align*} 
\int_{\R}\int_{\mathcal{D}_{n, k}} |\nabla \hat{z}_{\eps_n}|^2  + |\hat{z}_{\eps_n}|^2 \, dtdx \leq a_{n, l} &\leq c \, \mbox{exp}\left(\left(\left[\frac 12 \delta \eps_n^{-1}\right]-2\right) \ln{\theta}\right) \\
& \leq c\, \mbox{exp}\left(\left(\frac 14 \delta \eps_n^{-1}\right) \ln{\theta}\right),
\end{align*}
where $[ r ]$ denotes the integer part of a real number $r$. Thus we have finished the proof.
\end{proof}

\begin{lem} \label{concen}
For any $1 \leq k \leq m$, there holds that
$$
\lim_{\eps \to 0^+} \textnormal{dist}(\eps y_{\eps, k}, \, \mathcal{V})=0,
$$
where $\mathcal{V}$ is defined by \eqref{defv}.
\end{lem}
\begin{proof}
To prove this lemma, we argue by contradiction that there exist $1\leq k_0 \leq m$ and  a sequence $\{\eps_n\} \subset \R^+$ with $\eps_n=o_n(1)$ such that
$$
\lim_{n \to \infty} \mbox{dist}(\eps_n y_{\eps_n, k_0}, \, \mathcal{V}) >0.
$$
By Lemma \ref{sl}, we assume that $\eps_n y_{\eps_n, k_0} \to y_{k_0} \notin \mathcal{V}$ in $\R^N$ as $n \to \infty$,  then there is $\delta>0$ small such that, for any $n \in \mathbb{N}^+$ large,
\begin{align*}
\inf_{x \in B(y_{\eps_n, k_0},\, \delta \eps_n^{-1})} \nabla V(\eps_n x) \cdot \nabla V(\eps_n y_{\eps_n, k_0}) \geq \frac 12 |\nabla V(y_{k_0})|^2>0.
\end{align*}
Thus, for any $\tau \in [1-{2\eps_n/\delta}, \, 1+{2\eps_n/\delta}]$ and $n \in \mathbb{N}^+$ large,
\begin{align} \label{v}
\inf_{x \in B(y_{\eps_n, k_0},\, \tau \delta \eps_n^{-1})} \nabla V(\eps_n x) \cdot \nabla V(\eps_n y_{\eps_n, k_0}) \geq \frac 14 |\nabla V(y_{k_0})|^2>0.
\end{align}
We now set that
$$
{\bm \nu_n}:=\nabla V(\eps_n y_{\eps_n, k_0})=(\nu_{n, 1}, \nu_{n ,2}, \cdots, \nu_{n, N}), \quad w_{\eps_n}:=(v_{\eps_n}, \,u_{\eps_n} ).
$$
Recall that, for any $n \in \mathbb{N}^+$ large,
\begin{align} \label{equ}
Lz_{\eps_n} + V_{\eps_n}(x)z_{\eps_n}=f_{\eps_n}(x, |z_{\eps_n}|)z_{\eps_n}.
\end{align}
Taking the scalar product to \eqref{equ} with $\bm{\nu_{n}}\cdot \nabla w_{\eps_n} $ and integrating on $\R \times B(y_{\eps_n, k_0}, \, \tau\delta \eps_n^{-1})$, we then obtain that
\begin{align} \label{estimate1}
\begin{split}
&\int_{\R} \int_{B(y_{\eps_n, k_0}, \, \tau\delta \eps_n^{-1})} \left(Lz_{\eps_n} + V_{\eps_n}(x)z_{\eps_n}\right)\cdot \left({\bm \nu_{n}}\cdot \nabla w_{\eps_n}\right) \, dtdx \\
&=\int_{\R} \int_{B(y_{\eps_n, k_0}, \, \tau\delta \eps_n^{-1})} f_{\eps_n}(x, |z_{\eps_n}|) z_{\eps_n}\cdot \left({\bm \nu_{n}}\cdot \nabla w_{\eps_n} \right) \, dtdx.
\end{split}
\end{align}
In what follows, we shall calculate the terms in \eqref{estimate1} with the help of the divergence theorem. For the sake of convenience, let us introduce Einstein's summation convention on repeated indices. We assume that $1 \leq i \leq M$ and $1 \leq j \leq N$. Note first that
\begin{align*}
\int_{\R} \int_{B(y_{\eps_n, k_0}, \, \tau\delta \eps_n^{-1})} \partial_t u_{\eps_n} \cdot \left(\bm{\nu_n} \cdot \nabla v_{\eps_n} \right)\, dtdx
&=\int_{\R} \int_{B(y_{\eps_n, k_0}, \, \tau\delta \eps_n^{-1})} \partial_{t} u_{\eps_n, i} \, \partial_{j} v_{\eps_n, i} \, \nu_{n, j}\, dtdx \\
&=-\int_{\R} \int_{B(y_{\eps_n, k_0}, \, \tau\delta \eps_n^{-1})} \partial_{j} \partial_{t} v_{\eps_n, i} \, u_{\eps_n, i} \, \nu_{n, j}\, dtdx,
\end{align*}
from which we then get that
\begin{align*}
\int_{\R} \int_{B(y_{\eps_n, k_0}, \, \tau\delta \eps_n^{-1})} \partial_t u_{\eps_n} \cdot \left(\bm{\nu_n} \cdot \nabla v_{\eps_n} \right)\, dtdx
&=\int_{\R} \int_{B(y_{\eps_n, k_0}, \, \tau\delta \eps_n^{-1})}  \partial_{t} v_{\eps_n, i} \, \partial_{j}u_{\eps_n, i}\,  \nu_{n, j}\, dtdx \\
& \quad - \int_{\R} \int_{\partial B(y_{\eps_n, k_0}, \, \tau\delta \eps_n^{-1})} \partial_{t} v_{\eps_n, i} \, u_{\eps_n, i} \, \nu_{n, j} \, n_j\, dtdS,
\end{align*}
where $\bm{n}:=(n_1, n_2, \cdots, n_N)$ denotes the unit outward normal vector to $\partial B(y_{\eps_n, k_0}, \, \tau \delta \eps_n^{-1})$. As a consequence, we have that
\begin{align*} 
\int_{\R} \int_{B(y_{\eps_n, k_0}, \, \tau\delta \eps_n^{-1})} \partial_t u_{\eps_n} \cdot \left(\bm{\nu_n} \cdot \nabla v_{\eps_n} \right)\, dtdx -\int_{\R} \int_{B(y_{\eps_n, k_0}, \, \tau\delta \eps_n^{-1})} \partial_t v_{\eps_n} \cdot \left(\bm{\nu_n} \cdot \nabla u_{\eps_n} \right)\, dtdx
=I_1(\tau),
\end{align*}
where
\begin{align*}
 I_1(\tau):=- \int_{\R} \int_{\partial B(y_{\eps_n, k_0}, \, \tau \delta \eps_n^{-1})} \left( \partial_t v_{\eps_n} \cdot u_{\eps_n} \right) \left({\bm{\nu_n} \cdot \bm{n}}\right) \,dtdS.
\end{align*}
We next deal with the diffusion terms. By straightforward calculations, then
\begin{align}  \label{deltatm}
\begin{split}
&\int_{\R} \int_{B(y_{\eps_n, k_0}, \, \tau\delta \eps_n^{-1})} \Delta u_{\eps_n} \cdot \left(\bm{\nu_n} \cdot \nabla v_{\eps_n}\right) \, dtdx \\
&= \int_{\R} \int_{B(y_{\eps_n, k_0}, \, \tau\delta \eps_n^{-1})} \Delta u_{\eps_n, i} \, \partial_{j} v_{\eps_n, i} \, \nu_{n, j} \,dt dx \\
& = \int_{\R} \int_{\partial B(y_{\eps_n, k_0}, \, \tau\delta \eps_n^{-1})} \left(\nabla u_{\eps_n, i} \cdot \bm{n} \right) \,  \partial_{j} v_{\eps, i} \, \nu_{n, j} - \left( \nabla \partial_{j} v_{\eps_n, i} \cdot \bm{n}\ \right) \, u_{\eps_n, i} \, \nu_{n, j}  dt dS \\
& \quad + \int_{\R} \int_{B(y_{\eps_n, k_0}, \, \tau\delta \eps_n^{-1})} u_{\eps_n, i} \, \Delta \partial_{j} v_{\eps_n, i} \, \nu_{n, j} \, dt dx.
\end{split}
\end{align}
Observe that
\begin{align*}
&\int_{\R} \int_{B(y_{\eps_n, k_0}, \, \tau\delta \eps_n^{-1})} u_{\eps_n, i} \, \Delta \partial_{j}  v_{\eps_n, i} \, \nu_{n, j} \, dt dx \\
&= - \int_{\R} \int_{B(y_{\eps_n, k_0}, \, \tau\delta \eps_n^{-1})} \Delta v_{\eps_n, i}  \, \partial_{j} u_{\eps_n, i} \, \nu_{n, j}\, dt dx + \int_{\R} \int_{\partial B(y_{\eps_n, k_0}, \, \tau\delta \eps_n^{-1})} \Delta v_{\eps_n, i} \, u_{\eps_n, i} \, \nu_{n, j} \,  n_j \, dt dS.
\end{align*}
It then follows from \eqref{deltatm} that
\begin{align*} 
\int_{\R} \int_{B(y_{\eps_n, k_0}, \, \tau\delta \eps_n^{-1})} \Delta u_{\eps_n} \cdot \left(\bm{\nu_n} \cdot \nabla v_{\eps_n}\right) \, dtdx +\int_{\R} \int_{B(y_{\eps_n, k_0}, \, \tau\delta \eps_n^{-1})} \Delta v_{\eps_n} \cdot \left(\bm{\nu_n} \cdot \nabla u_{\eps_n}\right) \, dtdx
=I_2(\tau),
\end{align*}
where
\begin{align*}
I_2(\tau):&= \int_{\R} \int_{\partial B(y_{\eps_n, k_0}, \, \tau\delta \eps_n^{-1})} \left(\nabla u_{\eps_n, i} \cdot \bm{n} \right) \, \partial_{j} v_{\eps, i} \, \nu_{n, j}  - \left( \nabla \partial_{j} v_{\eps_n, i} \cdot \bm{n}\ \right) \, u_{\eps_n, i} \,  \nu_{n, j} \, dt dS\\
& \quad + \int_{\R} \int_{\partial B(y_{\eps_n, k_0}, \, \tau\delta \eps_n^{-1})} \left(\Delta v_{\eps_n} \cdot u_{\eps_n}\right)  \left(\bm{\nu}_n \cdot \bm{n} \right) \,dt dS.
\end{align*}
In addition, we can obtain that
\begin{align*} 
\int_{\R} \int_{B(y_{\eps_n, k_0}, \, \tau\delta \eps_n^{-1})} u_{\eps_n} \cdot \left(\bm{\nu} \cdot \nabla v_{\eps_n}\right) \, dtdx
+\int_{\R} \int_{B(y_{\eps_n, k_0}, \, \tau\delta \eps_n^{-1})} v_{\eps_n} \cdot \left(\bm{\nu} \cdot \nabla u_{\eps_n}\right) \, dtdx
=I_3(\tau),
\end{align*}
where
$$
I_3(\tau):=\int_{\R} \int_{\partial B(y_{\eps_n, k_0}, \, \tau\delta \eps_n^{-1})} \left(u_{\eps_n} \cdot v_{\eps_n}\right)  \left(\bm{\nu}_n \cdot \bm{n} \right)  \, dtdS.
$$
We are now ready to compute the potential terms. Notice that
\begin{align} \label{pterm1}
\begin{split}
&\int_{\R} \int_{B(y_{\eps_n, k_0}, \, \tau\delta \eps_n^{-1})} V_{\eps_n}(x) \, v_{\eps_n} \cdot \left(\bm{\nu} \cdot \nabla v_{\eps_n}\right) \, dtdx \\
&=\int_{\R}\int_{B(y_{\eps_n, k_0}, \, \tau\delta \eps_n^{-1})}  V_{\eps_n}(x)\, v_{\eps_n, i}  \, \partial_j v_{\eps_n, i} \, \nu_{n, j}\,dtdx \\
&= \frac 12  \int_{\R} \int_{B(y_{\eps_n, k_0}, \, \tau\delta \eps_n^{-1})} V_{\eps_n}(x)\, \partial_j \left( |v_{\eps_n, i}|^2\right) \, \nu_{n, j} \,dtdx  \\
&=-\frac {\eps_n}{2}\int_{\R} \int_{B(y_{\eps_n, k_0}, \, \tau\delta \eps_n^{-1})} \partial_j V(\eps_n x)\, |v_{\eps_n, i}|^2 \,\nu_{n, j} \,dtdx \\
&\quad + \frac 12 \int_{\R} \int_{\partial B(y_{\eps_n, k_0}, \, \tau\delta \eps_n^{-1})}  V_{\eps_n}(x)\, |v_{\eps_n, i}|^2\, \nu_{n, j} \, n_{j} \,dtdS.
\end{split}
\end{align}
Similarly, there holds that
\begin{align} \label{pterm2}
\begin{split}
&\int_{\R} \int_{B(y_{\eps_n, k_0}, \, \tau\delta \eps_n^{-1})} V_{\eps_n}(x) \, u_{\eps_n} \cdot \left(\bm{\nu} \cdot \nabla u_{\eps_n}\right) \, dtdx \\
&=-\frac {\eps_n}{2} \int_{\R} \int_{B(y_{\eps_n, k_0}, \, \tau\delta \eps_n^{-1})} \partial_j V(\eps_n x)\,|u_{\eps_n, i}|^2 \, \nu_{n, j} \,dtdx \\
&\quad + \frac 12 \int_{\R} \int_{\partial B(y_{\eps_n, k_0}, \, \tau\delta \eps_n^{-1})}  V_{\eps_n}(x)\, |u_{\eps_n, i}|^2 \, \nu_{n, j} \, n_{j} \,dtdS.
\end{split}
\end{align}
As a result of Lemma \ref{sl}, we know that $z_{\eps_n}(\cdot + \tau_{\eps_n, k_0},\, \cdot+ y_{\eps_n, k_0}) \wto z_{k_0} \neq 0$ in $E$ as $n \to \infty$. By Lemma \ref{embedding}, we then have that $z_{\eps_n}(\cdot + \tau_{\eps_n, k_0},\, \cdot+ y_{\eps_n, k_0}) \to z_{k_0}$ a.e. on $\R \times \R^N$ as $n \to \infty$. It then follows from \eqref{v} and Fatou's Lemma that, for any $n \in \mathbb{N}^+$ large,
\begin{align*}
& \int_{\R} \int_{B(y_{\eps_n, k_0}, \, \tau\delta \eps_n^{-1})} \partial_j V(\eps_n x) \, \left(|v_{\eps_n, i}|^2 + |u_{\eps_n, i}|^2 \right) \, \nu_{n, j}\,dtdx \\
&=\int_{\R} \int_{B(y_{\eps_n, k_0}, \, \tau\delta \eps_n^{-1})} \partial_j V(\eps_n x) \, \nu_{n, j}\, |z_{\eps_n}|^2 \,dtdx \\
&\geq \frac{|V(y_{k_0})|^2}{4}  \int_{\R} \int_{B(y_{\eps_n, k_0}, \, \tau\delta \eps_n^{-1})} |z_{\eps_n}|^2 \, dtdx \\
&=\frac{|V(y_{k_0})|^2}{4}  \int_{\R} \int_{B(0, \, \tau\delta \eps_n^{-1})} |z_{\eps_n}(t+ \tau_{\eps_n, k_0},\, x+ y_{\eps_n, k_0})|^2 \, dtdx \\
& \geq \frac{|V(y_{k_0})|^2}{8}  \int_{\R} \int_{\R^N} |z_{k_0}|^2 \, dtdx.
\end{align*}
Therefore, by using \eqref{pterm1} and \eqref{pterm2}, we get that, for any $n \in \mathbb{N}^+$ large,
\begin{align*} 
&\int_{\R} \int_{B(y_{\eps_n, k_0}, \, \tau\delta \eps_n^{-1})} V_{\eps_n}(x) \left( u_{\eps_n} \cdot \left(\bm{\nu} \cdot \nabla u_{\eps_n}\right) +  u_{\eps_n} \cdot \left(\bm{\nu} \cdot \nabla u_{\eps_n}\right) \right)\, dtdx \\
& \leq - \frac{\eps_n}{16} |V(y_{k_0})|^2 \int_{\R} \int_{\R^N} |z_{k_0}|^2 \, dtdx + I_4(\tau) ,
\end{align*}
where
$$
 I_4(\tau):=\frac 12\int_{\R} \int_{\partial B(y_{\eps_n, k_0}, \, \tau\delta \eps_n^{-1})}  V_{\eps_n}(x)\, |z_{\eps_n}|^2 \left(\bm{\nu}_n \cdot \bm{n} \right)\,dtdS.
$$
Finally, let us turn to treat the nonlinearity term. It is not difficult to see that
\begin{align*}
&\int_{\R} \int_{B(y_{\eps_n, k_0}, \, \tau\delta \eps_n^{-1})} f_{\eps_n}(x, |z_{\eps_n}|) z_{\eps_n}\cdot \left({\bm \nu_{n}}\cdot \nabla w_{\eps_n} \right) \, dtdx \\
&=\int_{\R} \int_{B(y_{\eps_n, k_0}, \, \tau\delta \eps_n^{-1})} f_{\eps_n}(x, |z_{\eps_n}|) \left(v_{\eps_n} \cdot \left(\bm{\nu} \cdot \nabla v_{\eps_n}\right) + u_{\eps_n} \cdot \left(\bm{\nu} \cdot \nabla u_{\eps_n}\right)\right) \, dtdx\\
&=\int_{\R} \int_{B(y_{\eps_n, k_0}, \, \tau\delta \eps_n^{-1})} f_{\eps_n}(x, |z_{\eps_n}|) \left(v_{\eps_n, i} \,\partial_j v_{\eps_n, i}+ u_{\eps_n, i} \,\partial_j u_{\eps_n, i}\right)\, \nu_{n ,j} \, dtdx \\
&= \int_{\R} \int_{B(y_{\eps_n, k_0}, \, \tau\delta \eps_n^{-1})}\left( \partial_j(F_{\eps_n}(x, |z_{\eps_n}|)) -\eps_n \partial_{j} F_x(\eps_n x, |z_{\eps_n}|)\right)\, \nu_{n, j} \, dtdx \\
&= \int_{\R} \int_{\partial B(y_{\eps_n, k_0}, \, \tau\delta \eps_n^{-1})} F_{\eps_n}(x, |z_{\eps_n}|) \,\nu_{n, j} \,n_j \, dtdS - \eps_n \int_{\R}\int_{B(y_{\eps_n, k_0}, \, \tau\delta \eps_n^{-1})}  \partial_{j} F_x(\eps_n x, |z_{\eps_n}|)\, \nu_{n, j} \, dtdx.
\end{align*}
Note that
\begin{align*}
\partial_{j} F_x(\eps_n x, |z_{\eps_n}|) \nu_{n, j}&= \partial_j \chi(\eps_n x) (\tilde{G}(|z_{\eps_n}|)-G(|z_{\eps_n}|)) \nu_{n, j} \\
&=\zeta'(\mbox{dist}(\eps_n x, \Lambda)) \,\partial_j \mbox{dist}(\eps_nx, \Lambda) \,\nu_{n, j} (\tilde{G}(|z_{\eps_n}|)-G(|z_{\eps_n}|)).
\end{align*}
If $ \eps_n x \in \Lambda$, then $\mbox{dist}(\eps_n x, \Lambda)=0$, this shows that $\partial_j F_x(\eps_n x, |z_{\eps_n}|) \nu_{n, j}=0$. If $\eps_n x  \in B(\eps_n y_{\eps_n,  k_0}, \,\tau \delta) \setminus \Lambda$, since $\tilde{G}(s) \leq G(s)$ and $\zeta'(s) \geq 0$ for any $s \geq 0$, it then yields from \eqref{vloc} that $ \partial_j F_x(\eps_n x, |z_{\eps_n}|)\,\nu_{n, j} \leq 0$.
Thus
\begin{align*} 
\int_{\R} \int_{B(y_{\eps_n, k_0}, \, \tau\delta \eps_n^{-1})} f_{\eps_n}(x, |z_{\eps_n}|) z_{\eps_n}\cdot \left({\bm \nu_{n}}\cdot \nabla w_{\eps_n} \right) \, dtdx  \geq  I_5(\tau),
\end{align*}
where
$$
I_5(\tau):=\int_{\R} \int_{\partial B(y_{\eps_n, k_0}, \, \tau\delta \eps_n^{-1})} F_{\eps_n}(x, |z_{\eps_n}|)  \left(\bm{\nu}_n \cdot \bm{n} \right)\, dtdS.
$$
From the arguments above, we then arrive at
\begin{align}\label{decaych}
I_1(\tau)-I_2(\tau)+I_3(\tau)+I_4(\tau)-I_5(\tau)\geq \frac{\eps_n}{16} |V(y_{k_0})|^2 \int_{\R} \int_{\R^N} |z_{k_0}|^2 \, dtdx.
\end{align}
Integrating \eqref{decaych} with respect to $\tau$ on $[1-{2\eps_n/\delta}, \, 1+{2\eps_n/\delta}]$, and applying \eqref{h1h2}, \eqref{equ},  Lemmas \ref{bddbr} and \ref{decay}, and H\"older's inequality,  we then deduce that there are $c>0$ and $C>0$ such that
$$
C \, \textnormal{exp}(-c\, \eps_n^{-1}) \geq \frac{\eps^2_n}{4 \delta} |V(y_{k_0})|^2 \int_{\R} \int_{\R^N} |z_{k_0}|^2,
$$
which is impossible for any $n \in \mathbb{N}^+$ large. Accordingly, the conclusion of the lemma holds, and the proof is completed.
\end{proof}

\begin{lem} \label{expdecay}
Let $\eps>0$ be small, then, for any $\delta>0$, there exist $c>0$ and $C>0$ such that
$$
|z_{\eps}(x, t)| \leq C \,\textnormal{exp}(-c \, \textnormal{dist} (x, \, (\mathcal{V}^{\delta})_{\eps})).
$$
\end{lem}
\begin{proof}
From Lemma \ref{concen}, for any $\eps>0$ small, we know that
$$
\mbox{dist}(\eps y_{\eps, k}, \, \R^N \setminus \mathcal{V}^{\delta}) \geq \frac{\delta}{2},
$$
which shows that
\begin{align*}
\mbox{dist}(y_{\eps, k}, \, \R^N \setminus (\mathcal{V}^{\delta})_{\eps}) \to \infty \,\,\mbox{as} \,\, \eps \to 0^+.
\end{align*}
Applying \eqref{bl}, and arguing as the proof of Lemma \ref{decay}, we have that
$$
\int_{\R} \int_{\mbox{dist}(x,\, \R^N \setminus (\mathcal{V}^{\delta})_{\eps}) \leq 1} |z_{\eps}(t, x)|^2 \, dtdx =o_{\eps}(1),
$$
from which and Corollary \ref{estimate} we are able to deduce that, for any $\gamma >0$, there exists $\tilde{\eps}>0$ such that, for any $0 <\eps < \tilde{\eps}$,
\begin{align*}
|z_{\eps}(t, x)| \leq \gamma \quad \mbox{for any}\,\, t \in \R, x \in  \R^N \setminus (\mathcal{V}^{\delta})_{\eps}.
\end{align*}
At this point, in order to complete the proof, it suffices to show that there is $R_0>0$ large such that
$$
|z_{\eps}(x, t)| \leq C \, \mbox{exp}(-c \, \mbox{dist}(x, \, (\mathcal{V}^{\delta})_{\eps})) \quad \mbox{for} \,\, \mbox{dist}(x,\, (\mathcal{V}^{\delta})_{\eps}) \geq R_0.
$$
To do this, we utilize the iteration arguments presented in the proof of Lemma \ref{decay}. For any $R \geq R_0$,  we define that
$$
\mathcal{V}_{\eps, l}:=\left\{x \in \R \setminus (\mathcal{V}^{\delta})_{\eps}: \mbox{dist}(x,\, (\mathcal{V}^{\delta})_{\eps})  \geq \frac {R}{2} + l\right\}.
$$
Let $\eta_l \in C^{\infty}(\R, [0, 1])$ be a cut-off function with $|\eta'(\tau)| \leq 4$ for any $\tau \in \R$, and
\begin{align*}
\eta_l(\tau):=\left\{
\begin{aligned}
&0,  \quad \tau \leq  \frac{R}{2}+l, \\
&1, \quad \tau \geq \frac{R}{2}+l +1.
\end{aligned}
\right.
\end{align*}
For any $x \in \R^N$,  we define that $\phi_{\eps, l}(x):=\eta_l(\mbox{dist}(x,\,(\mathcal{V}^{\delta})_{\eps}))$. Setting
$$
\hat{z}_{\eps}(t, x):=\left(u_{\eps}(t, x), v_{\eps}(-t, x)\right),
$$
we then know that
\begin{align}\label{hatzepsn1}
\partial_{t} \hat{z}_{\eps}-\Delta \hat{z}_{\eps} + \hat{z}_{\eps}=h,
\end{align}
where $h:=(h_1, h_2)$ with
$$
h_1(t, x):=-V_{\eps}(x)  v_{\eps}(t, x)+f_{\eps}(x, |z_{\eps}(t ,x)|)v_{\eps}(t, x),
$$
and
$$
h_2(t, x):=-V_{\eps}(x) u_{\eps}(-t, x) + f_{\eps}(x, |z_{\eps}(-t, x)|)u_{\eps}(-t, x).
$$
By taking the scalar product to \eqref{hatzepsn1} with $\phi_{\eps, l}^2 \hat{z}_{\eps}$, and integrating on $\R \times \R^N$, then
\begin{align*}
&\int_{\R}\int_{\R^N} \partial_t \hat{z}_{\eps} \cdot \hat{z}_{\eps} \phi_{\eps, l}^2 \, dtdx -\int_{\R} \int_{\R^N} \Delta \hat{z}_{\eps} \cdot \hat{z}_{\eps}\phi_{\eps, l}^2  \, dtdx + \int_{\R} \int_{\R^N} |\hat{z}_{\eps}|^2 \phi_{\eps, l}^2 \, dtdx \\
&= \int_{\R} \int_{\R^N} h \cdot \hat{z}_{\eps} \phi_{\eps, l}^2 \, dtdx.
\end{align*}
Using the same arguments as the proof of Lemma \ref{decay}, and letting $l=\left[R/2\right]-1$, we obtain that
$$
\int_{\R}\int_{\tilde{\mathcal{V}}} |\nabla \hat{z}_{\eps}|^2  + |\hat{z}_{\eps}|^2 \, dtdx \leq \tilde{C}\, \mbox{exp}\left( \frac{R}{3} \ln{\theta}\right),
$$
where $0<\theta<1$, and
$$
\tilde{\mathcal{V}}:=\left\{x \in \R \setminus (\mathcal{V}^{\delta})_{\eps}: \mbox{dist}(x,\, (\mathcal{V}^{\delta})_{\eps})  \geq R - 1\right\}.
$$
Thus, by Corollary \ref{estimate} and Lemma \ref{bddbr}, for any $R \geq R_0$ with $\mbox{dist}(x, \, (\mathcal{V}^{\delta})_{\eps})=R$,
\begin{align*}
|z_{\eps}(t, x)| \leq C \, \mbox{exp}\left( \frac{R}{3} \ln{\theta} \right)&=C \, \mbox{exp}\left( \frac{\ln{\theta}}{3} \,\mbox{dist}(x, \, (\mathcal{V}^{\delta})_{\eps})\right)\\
&=C \, \mbox{exp}\left(-c\,\mbox{dist}(x, \, (\mathcal{V}^{\delta})_{\eps})\right),
\end{align*}
where $c:=-\frac{\ln{\theta}}{3} $. Hence we have completed the proof.
\end{proof}

We are now in a position to establish Theorem \ref{theorem}.

\begin{proof}[Proof of Theorem \ref{theorem}] From $(V_2)$, we infer that $\mbox{dist}(\mathcal{V},\, \partial \Lambda) >0.$
For any $0<\delta <\mbox{dist}(\mathcal{V},\, \partial \Lambda)$, from Lemma \ref{expdecay}, we have that
\begin{align} \label{dec}
|z_{\eps}(t,x)| \leq C \, \mbox{exp} (-c \, \mbox{dist}(x, \, (\mathcal{V}^{\delta})_{\eps})).
\end{align}
If $t \in \R$ and $x \in \R^N \setminus \Lambda_{\eps}$, then $\mbox{dist}(x,\, (\mathcal{V}^{\delta})_{\eps}) \to \infty$ as $\eps \to 0^+$. Thus, for any $\eps>0$ small, it follows from \eqref{dec} that $g(z_{\eps}(t, x)) \leq \mu$ for any $t \in \R$ and $x \in \R^N \setminus \Lambda_{\eps}$. This in turn suggests that
$$
f_{\eps}(x, |z_{\eps}(t, x)|)=g(|z(t, x)|) \quad \mbox{for any} \, \, t \in \R, x \in \R^N \setminus \Lambda_{\eps}.
$$
If $t \in \R$  and $x \in \Lambda_{\eps}$, then $\chi(\eps x)=0$, which indicates that $f_{\eps}(x, |z(t, x)|)=g(|z(t, x)|)$. Hence, for any $\eps>0$ small,
$z_{\eps}$ is a ground state to \eqref{system1}. By making a change variable, from Lemma \ref{expdecay}, we obtain the decay of $z_{\eps}$. Thus the proof is completed.
\end{proof}

\par {\bf Acknowledgement}: The authors thank Dr. Quanguo Zhang for his valuable suggestions and comments on the manuscript.

\end{document}